\documentclass{amsart}

\usepackage{amsthm} 
\usepackage{color}

\usepackage{graphicx,epsfig}
\usepackage{epstopdf}
\usepackage{amsmath, amssymb, latexsym, euscript,amscd}
\usepackage{url}
\usepackage[all]{xy}
\usepackage{psfrag}
\usepackage[pagebackref=true]{hyperref}
\usepackage{mathrsfs}
\usepackage{tikz}
\usetikzlibrary{shapes.geometric}
\usepackage{subcaption}
\usepackage{multicol}

\usepackage{fullpage}

\usetikzlibrary{shapes.geometric}
\usetikzlibrary{positioning} 
\usepackage[alphabetic]{amsrefs}

\newcounter{notes}%

\definecolor{darkgreen}{rgb}{0.0, 0.5, 0.0}

\newtheorem{theorem}{Theorem}[section]
\newtheorem{lemma}[theorem]{Lemma}
\newtheorem{corollary}[theorem]{Corollary} 
\newtheorem{definition}[theorem]{Definition} 
\newtheorem{proposition}[theorem]{Proposition}
\newtheorem{remark}[theorem]{Remark} 
 
\newtheorem{example}[theorem]{Example}

\def\smallskip{\vspace{.15cm}}
\def\medskip{\vspace{.3cm}}
\def\text{\mbox}
\def\rh2{{\mathbb R}{\mathbb H}^2}
\def\ch2{{\mathbb C}{\mathbb H}^2}
\def\RR{{\mathbb R}}
\def\CC{{\mathbb C}}
\def\QQ{{\mathbb Q}}

\def\ZZ{{\mathbb Z}}
\def\NN{{\mathbb N}}

\def \PP{{\mathbb P}}

\def\RP2{{\mathbb{RP}}^2}
\def\RP3{{\mathbb{RP}}^3}
\def\RP{{\mathbb{RP}}}

\def\Id{\operatorname{Id}}

\def\SL{\operatorname{SL}}

\def\PSL{\operatorname{PSL}}

\def\GL{\operatorname{GL}}
\def\int{\operatorname{int}}

\def\mod{{\operatorname{mod}}}

\def\H2R{{\mathbb H}^2\times {\mathbb R}}

\def\Isom{\operatorname{Isom}}

\def\C2{\operatorname{C^2}}

\def\cube{\mathcal{Q}_3}
\def\square{\mathcal{Q}_2}

\def\Aut{\mathrm{Aut}}
\def\Ch{\mathrm{Ch}}
\def\Stab{\mathrm{Stab}}
\def\Fix{\mathrm{Fix}}

\def\CAT{\mathrm{CAT}}

\def\Min{\mathrm{Min}}
\def\bd{\partial}

\definecolor{back}{RGB}{255,255,255}
\definecolor{fore}{RGB}{0,0,0}
\definecolor{title}{RGB}{255,0,90}

\definecolor{green}{rgb}{0.0, 0.5, 0.0}
\definecolor{purple}{rgb}{0.5, 0.0, 0.5}
\definecolor{bluegreen}{rgb}{0.0,0.5, 0.5}
\definecolor{orange}{rgb}{1,0.5, 0.1}
\definecolor{redgreen}{rgb}{0.5, 0.5, 0.0}

\def\green{\color{green}}

\def\green{\color{green}}

\def\g2{{\green 2}}

\newcommand{\bv}{\left[\begin{array}{c}}
\newcommand{\ev}{\end{array}\right]}
\newcommand{\bbmat}{\begin{bmatrix}} 
\newcommand{\ebmat}{\end{bmatrix}}
\newcommand{\bmat}{\begin{matrix}} 
\newcommand{\emat}{\end{matrix}}
\newcommand{\bpmat}{\begin{pmatrix}} 
\newcommand{\epmat}{\end{pmatrix}}

\begin{document}
\title{Chabauty Limits of Subgroups of $\SL(n,\QQ_p)$} 
\author{Corina Ciobotaru, Arielle Leitner, and Alain Valette}
\thanks{corina.ciobotaru@gmail.com}\thanks{aleitner.math@gmail.com}\thanks{alain.valette@unine.ch}
\begin{abstract}

We study the Chabauty compactification of two families of closed subgroups of $\SL(n,\QQ_p)$.  The first family is the set of all parahoric subgroups of $\SL(n,\QQ_p)$.  Although the Chabauty compactification of parahoric subgroups is well studied, we give a different and more geometric proof using various Levi decompositions of $\SL(n,\QQ_p)$. 
Let $C$ be the subgroup of diagonal matrices in $\SL(n, \QQ_p)$.  The second family is the set of all $\SL(n,\QQ_p)$-conjugates of $C$.
We give a classification of the Chabauty limits of conjugates of $C$ using the action of $\SL(n,\QQ_p)$ on its associated Bruhat--Tits building and compute all of the limits for $n\leq 4$ (up to conjugacy). In contrast, for $n\geq 7$ we prove there are infinitely many $\SL(n,\QQ_p)$-nonconjugate Chabauty limits of conjugates of $C$. Along the way we construct an explicit homeomorphism between the Chabauty compactification in $\mathfrak{sl}(n, \QQ_p)$ of $\SL(n,\QQ_p)$-conjugates of the $p$-adic Lie algebra of $C$ and the Chabauty compactification of $\SL(n,\QQ_p)$-conjugates of $C$.
\end{abstract}

\maketitle 

\section{Introduction} 

For a locally compact topological space $X$, the set of all closed subsets of $X$ is a compact topological space with respect to the Chabauty topology (see Proposition \ref{prop::chabauty_conv}). Given a family $\mathcal{F}$ of closed subsets of $X$, it is natural to ask what is the closure of $\mathcal{F}$ with respect to the Chabauty topology, $\overline{\mathcal{F}}$, and whether or not closed subsets of $\overline {\mathcal{F}}$ satisfy the same properties as $\mathcal{F}$.  We call elements of $\overline{\mathcal{F}}$ the \textbf{Chabauty limits} of $\mathcal{F}$.


Let $X$ be a topological space admitting a continuous action by a locally compact group $G$, so the stabilizer $G_x$ in $G$ of every point $x \in X$ is a closed subgroup of $G$. Then the closure of $\{G_x\}_{x \in X}$ with respect to the Chabauty topology gives a natural compactification of $X$, called the \textbf{Chabauty compactification} of $X$. Examples include Riemannian symmetric spaces \cite{GJT}  and  Bruhat--Tits buildings \cite{GR} where points correspond to certain compact subgroups.


Let $G$ be a semisimple real Lie group with finite center and finitely many connected components, and let $K$ be the maximal compact subgroup.  Work of Guivarc'h--Ji--Taylor \cite{GJT} shows many different compactifications of $G/K$ are homeomorphic. Guivarc'h and R\' emy \cite{GR} extend many of these results to Bruhat--Tits buildings  using a probabilistic method which holds for general semi-simple algebraic groups $G$ over non Archimedean local fields.  Caprace and L\'{e}cureux \cite{CL} generalize these results to a larger class of buildings. 

In Section \ref{Chabauty_parahoric} we give a different and more geometric proof of \cite[Theorem 3, Corollary 4]{GR} and~\cite[Theorem 3.14]{Htt} for $\SL(n, \QQ_p)$ using various Levi decompositions of $\SL(n, \QQ_p)$. This geometric method works mainly for $\SL(n, \QQ_p)$ and the reason is given after the proof of Theorem \ref{thm::main_thm}.  The various Levi decompositions used for $\SL(n, \QQ_p)$ are recalled and proved in Section \ref{sec::sec_three} for general closed non-compact subgroups acting 
on locally finite thick affine buildings. Although the main part of Section \ref{sec::sec_three} is a brief overview of groups acting on affine buildings, the last subsection is new (see Propositions \ref{prop::res_building}, \ref{prop::res_building_action}, \ref{prop::res_building_str_tran} and Lemma \ref{rem::SL_Levi_factor}) and generalizes the notion of panel trees.

\medskip
For a local field $\mathbb{F}$, we study a second family of closed subgroups of $\SL(n,\mathbb{F})$. Let $C$ be the set of all diagonal matrices of $\SL(n,\mathbb{F})$, also called the \textbf{diagonal Cartan subgroup}.  Haettel \cite{Htt_2} studies the Chabauty compactification of all conjugates of $C$ in $\SL(n,\RR)$, and Iliev and Manivel \cite{IM} study that compactification for $\SL(n, \mathbb{C})$.  In the second part of this paper we study the Chabauty compactification $\overline{Cart(\SL(n,\QQ_p))}^{Ch}$ of all conjugates of $C$ in $\SL(n,\QQ_p)$. Using the action of $\SL(n,\QQ_p)$ on its associated Bruhat--Tits building $X$ we describe Chabauty limits of $C$ up to conjugacy, and study how the different geometries produced by $C$ and its limits change. 

\begin{theorem}(See Theorems \ref{ellip_unip}, \ref{thm::hyper_Cartan},  \ref{rem::hyp_Cartan})
\label{hyp_ellip}  Let $H \in \overline{Cart(\SL(n,\QQ_p))}^{Ch}$.  If $H$ contains hyperbolic elements, then $H$ stabilizes a flat in $X$.  If $H$ contains no hyperbolic elements, then up to conjugacy $H$ is contained in $U \cdot \mu_n$, where $U$ is the unipotent radical of the Borel subgroup of $\SL(n,\QQ_p)$, and $\mu_n$ is the group of $n$-th roots of unity in $\QQ_p^*$.
\end{theorem} 


To prove the second part of Theorem \ref{hyp_ellip} we need to build a bijection between subgroups and subalgebras, which we do in section \ref{gr_alg}.  We construct an explicit bijection between  $\overline{Cart(\SL(n,\QQ_p))}^{Ch}$ and the Chabauty compactification $\overline{Cart(\mathfrak{sl(n, \QQ_p)})}^{Ch}$ in $\mathfrak{sl(n, \QQ_p)}$ of all $\SL(n,\QQ_p)$-conjugates of the Lie subalgebra $\mathfrak{c} \subset \mathfrak{sl}(n, \QQ_p)$ of $C$.  Although there is a Lie functor that has been well studied in \cite{Bou75, Ser92} and elsewhere,  this is functorial only to a neighbourhood (the \emph{germ}) of the identity element of the Lie group, and for the $p$-adic case the germ is very small, and does not give a bijection between Lie groups and Lie algebras.  It might be possible that  a $p$-adic Lie algebra to be the Lie algebra of two different $p$-adic Lie groups that share the same group germ  (for more see the introduction of Section \ref{homeo_Gr}). The map we produce is a bijection, and  surprisingly, it is continuous with respect to the Chabauty topology.   In general, (canonical) maps between spaces endowed with the Chabauty topology are only upper/lower semi-continuous.

\begin{theorem}(See Propositions \ref{welldef}, \ref{surj}, \ref{inj}, Corollary \ref{cor::homeo})
\label{grbij} 
The map 
$$Gr: \overline{Cart(\mathfrak{sl(n, \QQ_p)})}^{Ch} \to \overline{Cart(\SL(n,\QQ_p))}^{Ch}$$
 $$A  \mapsto Gr(A) := \langle A, \Id \rangle \cap \SL(n,\QQ_p)$$ is a homeomorphism. 
\end{theorem} 

Regarding the algebras in $\overline{Cart (\mathfrak{sl(n, \QQ_p)})}^{Ch}$ we also have:

\begin{proposition}(See Corollary \ref{cor::dim_same})
Every element of  $\overline{Cart (\mathfrak{sl(n, \QQ_p)})}^{Ch}$ is an abelian subalgebra with respect to the Lie bracket of $\mathfrak{sl(n, \QQ_p)}$ and of dimension $n-1$.
\end{proposition}

\medskip
Chabauty limits of $C$ in $\SL(n, \mathbb{R})$  are classified in \cite{Leitner, Leitnersl3, Htt_2, IM}.  We extend these results and describe conjugacy classes of limits of $C$ in $\SL(n, \QQ_p)$ for $n \leq 4$ and show there are finitely many: 

\begin{theorem}(See Section \ref{lowdim})
\label{lowdimthm} 
\begin{itemize} 
 \item 
 There are two Chabauty limits of $C$ up to conjugacy in $\SL(2, \QQ_p)$. 
 \item 
 There are 4 + $\cube$ Chabauty limits of $C$ up to conjugacy in $\SL(3, \QQ_p)$. 
 \item 
 There are at least $12 + \square +\mathcal{Q}_4$ Chabauty limits of $C$ up to conjugacy in $\SL(4, \QQ_p)$ and less than $12 + \square + \mathcal{Q}_8$. 
 \end{itemize} 
 Where $\mathcal{Q}_k= | \QQ_p ^* / {\QQ_p ^ * } ^k|$ are computed in Section \ref{lowdim} .  
\end{theorem}

We explicitly write down all these Chabauty limits in Section \ref{lowdim}. 
Notice there are more conjugacy classes of limits over $\mathbb{Q}_p$ than over $\mathbb{R}$, since the Galois group of the algebraic closure over $\QQ_p$ is bigger. Following ideas in \cite{Leitner, IM} we show in Section \ref{dim_cl} that for $n \geq 5$, \textbf{not} all abelian subalgebras of the dimension $n-1$ are Chabauty limits of $\mathfrak{c}$.  Moreover, for $n \geq 7$ we show there are infinitely many nonconjugate Chabauty limits of $C$.  For $n=5,6$ it is unknown whether there are finitely or infinitely many nonconjugate Chabauty limits of $C$ in $\SL(n, \QQ_p)$. 

Acknowledgements : We thank Uri Bader, Marc Burger, Yair Glasner, Thomas Haettel, Tobias Hartnick and Guy Rousseau for illuminating conversations. Leitner was partially supported by the ISF-UGC joint research program framework grant No. 1469/14 and No. 577/15.  

\section{Chabauty topology}

A good introduction to Chabauty topology~\cite{Ch} can be found in~\cite{CoPau} or~\cite[Section 2]{Htt} and the references therein. We recall briefly some facts that are used in this paper.

\medskip
For a locally compact topological space $X$ we denote by $\mathcal{F}(X)$ the set of all closed subsets of $X$. $\mathcal{F}(X)$ is endowed with the Chabauty topology where every open set is a union of finite intersections of subsets of the form $O_K:=\{ F \in \mathcal{F}(X) \; \vert \; F \cap K =\emptyset\}$, where $K$ is a compact  subset of $X$, or $O'_U:=\{ F \in \mathcal{F}(X) \; \vert \; F \cap U \neq \emptyset\}$, where $U$ is an open subset of $X$. By \cite[Proposition~1.7, p.~58]{CoPau} the space $\mathcal{F}(X)$ is compact with respect to the Chabauty topology. Moreover, if $X$ is Hausdorff and second countable then $\mathcal{F}(X)$ is separable and metrisable, thus Hausdorff (see \cite[Proposition I.3.1.2]{CEM}).

\begin{proposition}(\cite[Proposition~1.8, p.~60]{CoPau}, \cite[Proposition I.3.1.3]{CEM})
\label{prop::chabauty_conv}
 Suppose $X$ is a locally compact metric space.
A sequence of closed subsets $\{F_n\}_{n \in \NN} \subset \mathcal{F}(X)$  converges to $F \in \mathcal{F}(X)$ if and only if the following two conditions are satisfied:
\begin{itemize} 
\item[1)] For every $f \in F$ there is a sequence $\{f_n \in F_n\}_{n \in \NN}$ converging to $f$;
\item[2)] For every sequence $\{f_n \in F_n\}_{n \in \NN}$, if there is a strictly increasing subsequence $\{n_k\}_{k \in \NN}$ such that $\{f_{n_k} \in F_{n_k}\}_{k \in \NN}$ converges to $f$, then $f \in F$.
\end{itemize}
\end{proposition}

For a locally compact group $G$ we denote by $\mathcal{S}(G)$ the set of all closed subgroups of $G$. By \cite[Proposition~1.7, p.~58]{CoPau} the space $\mathcal{S}(G)$ is closed in $\mathcal{F}(G)$, with respect to the Chabauty topology, and thus compact. Moreover, Proposition~\ref{prop::chabauty_conv} can be applied to a sequence of closed subgroups $\{H_n\}_{n \in \NN} \subset \mathcal{S}(G)$ converging to $H \in \mathcal{S}(G)$, obtaining a similar characterisation of convergence in $\mathcal{S}(G)$.

We thank Daryl Cooper for suggesting the next idea.  A group, $H$, satisfies an \textbf{universal relation} if there is a finitely generated free group $F$ and a word $w \in F$ such that for all homomorphisms $\theta: F \to H$ we have $\theta(w) =1$.  For example, an abelian group satisfies $xyx^{-1}y^{-1}=1$. 
\begin{proposition}  
\label{prop::Cooper}
If $H \leq G$ satisfies a universal relation, $w$, then so does every $G$-conjugacy limit $L$ of $H$. 
\end{proposition} 
\begin{proof}   
The main idea of the proof follows as in Bridson--de la Harpe--Kleptsyn~\cite[Proposition 3.4 i)]{BHK}.
Let $\{g_n\}_k \subset G$ be such that $\{g_k H g_k^{-1}\}_k$ converges to $L$ in the Chabauty topology.
Suppose for contradiction that $L$ does not satisfy the universal relation of $H$.  Then for every finitely generated free group $F$ and for every word in $ v \in F$ there is a homomorphism $\theta : F \to L$ such that $\theta(v) \neq 1$. In particular, take the free group $F$ and the word $w \in F$ given by the universal relation of $H$. Then there exist $l_1,..., l_n \in L$ such that $w (l_1, ..., l_n) \neq 1$.  Moreover, as the multiplication in $G$ is continuous, there exist neighborhoods of the identity, $U_i$ for each $i \in \{1, \cdots, n\}$, so that for all $l'_i \in l_i \cdot U_i$ then $w (l_1 ', ..., l_n ' ) \neq 1$.  Set $U = \cap _{i =1 } ^n U_i$, and consider the neighborhood of $L$ in $ \mathcal{S}(G)$
$$V= V_{\{l_1, \cdots, l_n\}, U^{-1}}(L):= \{D \in \mathcal{S}(G) \; | \; D \cap \{l_1, \cdots, l_n\} \subset L \cdot U ^{-1} , L \cap \{ l_1, \cdots,l_n\} \subset D \cdot U^{-1} \}.  $$
Then for all $D \in V$ we have $\{l_1, \cdots, l_n \} \subset D \cdot U^{-1}$.  In particular, there exists $l_i ' \in D$ for $1 \leq i \leq n$ such that $\l_i ' \in l_i U_i$.  But every element $l_i \in L$ is approximated by a sequence of elements from $\{g_k H g_k^{-1}\}_k$. Thus, as $k \to \infty$, $g_kHg_k^{-1}$ does not satisfy the universal relation, which gives the necessary contradiction. 
\end{proof}

\section{On affine buildings}\label{intro_bldg}

When working with a semi-simple algebraic group $G$ over a non-Archimedean local field (e.g., the $p$-adic field $\QQ_{p}$) its associated ``symmetric space''  is the Bruhat--Tits building that is a locally finite thick affine building (see~\cite{AB, Gar97, Ron89}). 

In the Davis' realisation every locally finite affine building $\Delta$ is a $\CAT(0)$--space  and  one can define the visual boundary of $\Delta$, denoted by $\partial \Delta$. By the theory of $\CAT(0)$--spaces, the space $\Delta \cup \partial \Delta$ admits a cone topology (for definitions see~\cite[Part II, Section~8]{BH99}) and it is a compact space. Moreover, $\partial \Delta$ is a spherical building, called the \textbf{spherical building at infinity of $\Delta$}. We denote by $\Aut(\Delta)$ the group of all automorphisms of $\Delta$. The group $\Aut(\Delta)$ is endowed with the \textbf{compact-open topology};  
and it is a totally disconnected locally compact group (e.g., the pointwise stabilizer in $\Aut(\Delta)$ of a finite number of points of $\Delta$ is compact and open). The set of all such pointwise stabilizers forms a basis for the compact-open topology on $\Aut(\Delta)$: Let $x \in \Delta$ be a point. A sequence $\{g_n\}_{n \in \NN} \subset \Aut(\Delta)$ converges to the  automorphism $g \in \Aut(\Delta)$ if  for every $r >0$ there exists $N_r>0$ such that $g_n(y)=g(y)$ for every $y \in B(x,r)$ and every $n \geq N_r$. 

Recall a closed subgroup $ G \leq \Aut(\Delta)$ acts continuously on $\Delta \cup \partial \Delta$. This means for every $\{x_n\}_{n \in \NN} \subset \Delta$ that converges to $x \in \Delta \cup \partial \Delta$ (with respect to the cone topology on $\Delta \cup \partial \Delta$) and every $\{g_n\}_{n \in \NN} \subset G$ that converges to $g \in G$ (with respect to the group topology on $G$) we have $\{g_n(x_n)\}_{n \in \NN}$ converges to $g(x)$, with respect to the cone topology on $\Delta \cup \partial \Delta$.

\begin{definition}
\label{def::str_trans}
Let $(\Delta, \mathcal{A})$ be a building (e.g., a spherical building or an affine building) and $G \leq \Aut(\Delta)$. We say that $G$ acts \textbf{strongly transitivly} on $\Delta$ if for any two pairs $(\Sigma_1,c_{1})$ and $(\Sigma_2,c_{2})$ consisting of apartments $\Sigma_i \in \mathcal{A}$ and chambers $c_i$ of  $\Sigma_i$, there exists $g \in G$ such that $g(\Sigma_1)=\Sigma_2$, $g(c_{1})=c_{2}$ and by preserving the types of the vertices of $c_1$ and $c_2$. Recall buildings are colorable, i.e., the vertices of any chamber are colored differently and any chamber uses the same set of colors. We say that $g \in \Aut(\Delta)$ is \textbf{type-preserving} if $g$ preserves the coloration of the building. We say that $G \leq \Aut(\Delta)$ is \textbf{type-preserving} if all elements of $G$ are type-preserving.
\end{definition}

The Bruhat--Tits building $X$ associated with $\SL(n, \QQ_p)$ is constructed in \cite[Chapter~19]{Gar97}, for example, and it is considered with its maximal system of apartments. The building $X$ is locally finite, thick, affine, of type $\tilde{A}_{n-1}$, and of dimension $n-1$.

\begin{example}
For $\SL(2, \QQ_p)$ the Bruhat--Tits building $X$ is a $(p+1)$--regular tree. It is an affine building of dimension 1 where an  apartment is a bi-infinite geodesic line of the tree and a chamber is an edge.
\end{example}

The group $\SL(n, \QQ_p)$ is a closed subgroup of $\Aut(X)$ that acts strongly transitively and by type-preserving automorphisms on $X$. Also $\SL(n, \QQ_p) $ acts strongly transitively on its spherical building at infinity $\partial X$.

\begin{definition}
Let $(X,d)$ be a $\CAT(0)$-space, and let $\gamma$ be an isometry of $X$. Let 
$$\Min(\gamma):=\{x \in X \; | \; d(x, \gamma(x))=|\gamma| \},$$ 
where $|\gamma|:=\inf_{x \in X}\{d(x, \gamma(x))\}$ denotes the \textbf{translation length} of $\gamma$. If $\Min(\gamma)$ is the empty set then $\gamma$ is \textbf{ parabolic}. If $\Min(\gamma)$ is not the empty set then $\gamma$ is \textbf{ elliptic} if $\vert \gamma \vert=0$, otherwise $\gamma$ is \textbf{ hyperbolic}. A semi-simple isometry of $X$ is by definition either elliptic or hyperbolic.
\end{definition}

When $\Delta$ is an affine building, any element of $\Aut(\Delta)$ is either elliptic or hyperbolic.

\begin{lemma}
\label{lem::same_trans_length}
Let $\Delta$ be a locally finite affine building and $G$ be a closed non-compact subgroup of $\Aut(\Delta)$. Suppose the sequence $\{g_n\}_{n \in \NN} \subset G$ converges to an isometry $g \in G$.  If $g$ is elliptic (resp., hyperbolic) then there exists $N>0$ such that $g_n$ is elliptic (resp., hyperbolic) for every $n \geq N$. In particular, $\vert g \vert= \vert g_n\vert$ for every $n\geq N$.
\end{lemma}

\begin{proof}
As $g \in \Aut(\Delta)$, either $g$ is elliptic or $g$ is hyperbolic.

Suppose $g$ is elliptic. Then there exists $x \in \Min(g) \subset \Delta$, thus $g(x)=x$. By considering a ball $B(x,r)$ around $x$ with $r>0$ there exists $N_r$ with the required properties.

Suppose  $g$ is hyperbolic. Then $ \Min(g) $ is not empty and $g$-invariant ($\Min(g) =g \Min(g) $). In particular, for every $x \in \Min(g)$ we have  $g^{n}(x) \in \Min(g)$, for every $n \in \ZZ$. Moreover,  $\{g^{n}(x)\}_{n \in \ZZ} $ is contained in an apartment of $\Delta$ (that is not necessarily unique) and the points $\{g^{n}(x)\}_{n \in \ZZ} $ are on a bi-infinite (Euclidean) line, that forms a translation axis of $g$. Thus by taking a ball $B(x,r)$ with $r$ big enough such that $g^{\pm4}(x) \in B(x,r)$, then there exists $N_r>0$ such that $g_n(y)=g(y)$ for every $y \in B(x,r)$ and every $n >N_r$. In particular, the oriented segment $[x,g(x)]$ is sent by every $g_n$, with $n >N_r$, to the segment $[g(x),g^2(x)]$, preserving the orientation. Moreover, as $[x,g(x)]$ and $[g(x),g^2(x)]$ are collinear and because the orientation is not reversed, by~\cite[Lemma 2.8]{CaCi} we have $g_{n}$ is a hyperbolic element and $\vert g_n \vert =\vert g\vert $, for every $n >N_r$.
\end{proof}

\section{Decompositions of groups acting on affine buildings}
\label{sec::sec_three}

Let $\Delta$ be a locally finite thick affine building with the complete system of apartments and let $G$ be a closed subgroup of $\Aut(\Delta)$ that acts strongly transitively and by type-preserving automorphisms on $\Delta$. A \textbf{good maximal compact} subgroup of $G$ is the stabilizer $\Stab_{G}(x):=\{ g \in G \; \vert \; g(x)=x\}$ of a special vertex $x$ of $\Delta$. (For the definition of a special vertex see~\cite[Sec.~16.1]{Gar97}.)

 For the rest of the article we fix an apartment $\Sigma$ of $\Delta$ and a special vertex $x \in \Sigma$ and let $K:=\Stab_{G}(x)$. The \textbf{Cartan subgroup $C$} of $G$ is the centre of the subgroup $\Stab_{G}(\Sigma):=\{ g \in G \; \vert \; g(\Sigma)=\Sigma\}$. Since $G$ acts transitively on the set of all apartments of $\Delta$, the subgroup $C$ is unique up to conjugation in $G$.  Let $\Fix_{G}(\Sigma):=\{ g \in G \; \vert \; g(z)=z, \forall z \in \Sigma\}$.  Let $\Ch(\partial \Delta)$, resp., $\Ch(\partial \Sigma)$, be the set of all chambers of the spherical building $\partial \Delta$, resp., spherical apartment $\partial \Sigma$.

Let $c \in \Ch(\partial \Sigma) \subset \Ch(\partial \Delta)$ be an ideal chamber  of $\partial \Sigma$, that is fixed for what follows. The \textbf{Borel subgroup $B$} of $G$ is the closed subgroup $B:=\Stab_G(c):=\{ g \in G \; \vert \; g(c)=c\}$. Since $G$ acts on $\Delta$ strongly transitively, and induces a strongly transitively action on the visual boundary $\partial \Delta$, the Borel subgroup $B$ is unique up to conjugation in $G$. Let $B^0:=\{ g \in G \; \vert \; g(c)=c, \quad g \text{ elliptic}\}$. 

Recall $$W^{\text{\tiny{aff}}}=\Stab_{G}(\Sigma) / \Fix_{G}(\Sigma)$$ is the \textbf{affine Weyl group} associated with the affine building $\Delta$; $(W^{\text{\tiny{aff}}},S^{\text{\tiny{aff}}})$ is the affine Coxeter system of $\Delta$ and $S^{\text{\tiny{aff}}}$ is the set of the reflexions through the walls of a chamber in $\Ch(\Sigma)$. Another well-known fact is that $W^{\text{\tiny{aff}}}$ contains a maximal abelian normal subgroup isomorphic to $\mathbb{Z}^{m}$, whose elements are Euclidean translation automorphisms of the apartment $\Sigma$ and $m$ is the Euclidean dimension of $\Sigma$. Denote this maximal abelian normal subgroup by $A$ and note its elements are images of hyperbolic automorphisms of $\Delta$. In particular, every element of $A$ can be lifted (not in a unique way) to a hyperbolic element of $G$. 

Note for a general closed strongly transitive and type-preserving subgroup $G$ of $\Aut(\Delta)$ its corresponding abelian subgroup $A < W^{\text{\tiny{aff}}}$ does not necessarily lift to an abelian subgroup of $G$. Still, this is the case if we consider $G$ to be a semi-simple algebraic group over a non Archimedean local field (e.g., $G=\SL(n, \QQ_p)$). By the theory of Bruhat--Tits, the abelian group $A <W^{\text{\tiny{aff}}}$ is in fact the restriction to $\Sigma$ of a maximal split torus of $G$, which is abelian.

 Note the subgroups $K, C, A, B$ all depend on the choice of the apartment $\Sigma$, the special vertex $x \in \Sigma$ and the ideal chamber $c \in \Ch(\partial \Sigma)$. In particular, the decompositions below also depend on the choice of $\Sigma, x$ and $c$. 

\begin{lemma}(\textbf{Cartan decomposition}, for a proof see~\cite[Lemma~4.7]{Cio})
\label{lem::polar_decom}
Let $G$ be a strongly transitive and type-preserving subgroup of $\Aut(\Delta)$,  that is not necessarily closed. Then $G=K A K$.
\end{lemma}

\begin{lemma}(\textbf{Iwasawa decomposition}, for a proof see~\cite[p.~14]{Cio})
\label{lem::Iwasawa_decom}
Let $G$ be a strongly transitive and type-preserving subgroup of $\Aut(\Delta)$,  that is not necessarily closed. Then $G=K A B^0$ and $B=AB^0$.
\end{lemma}

\begin{remark} 
\label{rem::special_vertex}
For a general affine building $\Delta$  and $G$ a strongly transitive and type-preserving subgroup of $\Aut(\Delta)$ (not necessarily closed)  the Cartan and Iwasawa decompositions above do \emph{not} hold for $x$ a non-special vertex and $K=\Stab_{G}(x)$. If $x$ is not special, the translation subgroup of the affine Weyl group of $G$ might not act transitively on the set of (non-special) vertices of the same type, which is an ingredient in the proofs of the Cartan and Iwasawa decompositions. 
\end{remark}

The \textbf{Weyl group} associated with the affine building $\Delta$ is the group $$W=\Stab_{G}(\Sigma) / (\Stab_{G}(\Sigma) \cap B).$$

Here $(W,S)$ is the finite Coxeter system of $ \partial \Delta$ and $S$ is the set of all the reflexions through the walls in  $\Sigma$ that determine the chamber $c \in \Ch( \partial \Sigma)$  fixing the vertex $x$.  So $W$ is a finite group, in contrast to the infinite $W^{\text{\tiny{aff}}}$. 

 For a simplex $\sigma$ of $c$ there is a unique Coxeter subsystem $(W_I, I)$ of $(W,S)$ such that $\{s \in S \; \vert \; s(\sigma)=\sigma\}=I$; let $P^{I}:=G_{\sigma}:=\{ g \in G \; \vert \; g(\sigma)=\sigma\}$. For $\sigma=c$ the corresponding $(W_I, I)=(\Id, \emptyset)$ and $P^{\emptyset}=B$. Note  $B \subset P^{I}$, for every $(W_I,I)$ Coxeter subsystem of $(W,S)$ and $P^{I}$ is unique up to conjugation in $G$. The subgroups $P^{I}$ are closed subgroups of $G$ and are called the \textbf{parabolic subgroups} of $G$. The Borel subgroup $B$ is the minimal parabolic among them. 
 
For an ideal chamber $c \in \Ch( \partial \Sigma)$ there is a unique ideal chamber $c' \in \Ch( \partial \Sigma)$ opposite $c$. Every simplex $\sigma$ of $c$ has a unique opposite simplex $\sigma'$  in $\partial \Sigma$, with $\sigma' \subset c'$. The Coxeter subsystem $(W_I,I)$ associated with $\sigma'$ equals the one of $\sigma$. Still, the corresponding parabolic subgroups are not the same. To avoid the confusion, when $\sigma$ and $\sigma'$ are opposite ideal simplices in $\partial \Delta$ we denote $P^{I}_{+}:= G_{\sigma}=\{ g \in G \; \vert \; g(\sigma)=\sigma\}$ and $P^{I}_{-}:=G_{\sigma'}=\{ g \in G \; \vert \; g(\sigma')=\sigma'\}$. Correspondingly, for opposite ideal chambers we use the notation $B_{+}=G_{c}, B_{-}=G_{c'}$. We also use the notation $G_{\sigma}^{0}:=\{ g \in G \; \vert \; g(\sigma)=\sigma,  \quad g \text{ elliptic}\}$.

\begin{definition}[See Section 3 of \cite{BW04}]
\label{def::contr_groups}
Let $G $ be a totally disconnected locally compact group and take $a \in G$.  The \textbf{parabolic}  and \textbf{contraction} subgroups associated to $a$ are
\begin{equation} 
\label{equ::parab_subgroup}
P^{+}_{a}:=\{ g \in G \; \vert \; \{a^{-n}g a^{n}\}_{n \in \mathbb{N}} \text{ is bounded}\}, \qquad U^{+}_{a}:=\{ g \in G \; \vert \; \lim_{n \to \infty }a^{-n}g a^{n}=e\}.
\end{equation}
The group $P^+_a$ is a closed subgroup of $G$, but in general  $U^{+}_{a}$ is not closed.  Similarly, we define $P^{-}_{a}$ and $U^{-}_{a}$ using $a^{n}g a^{-n}$. \end{definition} 

\begin{example}
We apply Definition~\ref{def::contr_groups} to a hyperbolic element $a \in G$, where $G \leq \Aut(\Delta)$ and $\Delta$ is a locally finite affine building. Every hyperbolic element $a \in \Aut(\Delta)$ admits unique attracting and repelling endpoints $\xi_{+}$ and $\xi_{-}$ contained in $\partial \Delta$. 
Then $\xi_{+}$ and $\xi_{-}$ are the endpoints of the translation axis of $a$.

\end{example}

\begin{proposition}(\textbf{Levi decomposition}, for a proof see~\cite[Prop.~4.15 and Cor.~4.17]{Cio})
\label{coro::geom_levi_decom}
Let $G$ be a closed, strongly transitive and type-preserving subgroup of $\Aut(\Delta)$. Let $\Sigma$ be an apartment in $\Delta$ and assume there is a hyperbolic automorphism $a \in \Stab_{G}(\Sigma)$. Let $\xi_{+}, \xi_{-}$  be the attracting and repelling endpoints of $a$ and $\sigma_{+},\sigma_{-}$ the unique simplices in $\partial \Sigma$ whose interior contain $\xi_{+}$, respectively $\xi_{-}$. Note $\sigma_{+}$ is opposite $\sigma_{-}$ and let $(W_{I},I)$ be the corresponding Coxeter subsystem of $\sigma_{+}$.

Then $P^{\pm}_{a}=G_{\sigma_{\pm}}=P^{I}_{\pm}= U^{\pm}_{a} M_{I}$, where $M_{I}:= G_{\sigma_{+}} \cap  G_{\sigma_{-}}$, and $U^{\pm}_{a} $ is normal in  $G_{\sigma_{\pm}}^{0}$ and $G_{\sigma_{\pm}}$. In particular,  $G_{\sigma_{\pm}}^{0}= U^{\pm}_{a} (M_{I} \cap G_{\sigma_{\pm}}^{0})$.
\end{proposition}

 \begin{remark}
\label{rem::uni_rad}
Let $a,b \in \Stab_{G}(\Sigma)$ be two hyperbolic elements  with attracting endpoints in the interior of the same  unique simplex $\sigma_{+}$ in $\partial \Sigma$. By Proposition~\ref{coro::geom_levi_decom} one can see that $U^{+}_{a}=U^{+}_{b}$. Thus we define $U^{I}_{+}:=U^{+}_{a}$ and call it the \textbf{unipotent radical} of $P^{I}=G_{\sigma_{+}}$. Note $U^{I}$ depends only on the simplex $\sigma_{+}$ and not on the chosen hyperbolic element $a$. The subgroup $M_I$ is called the \textbf{Levi factor/subgroup} of $P^{I}$. 
\end{remark}

\begin{remark}
\label{rem::shape_hyp_elements}
Let $\mathbb{G}$ be a semi-simple algebraic group over a non Archimedean local field $k$ and let $G=\mathbb{G}(k)$ be the group of all $k$--rational points of $\mathbb{G}$ over the completion of $k$. For example, $G=\SL(n, \QQ_p)$.  Denote by $\Delta$ the corresponding locally finite thick affine building of dimension $m$ on which $G$ acts by type-preserving automorphisms and strongly transitively.  Let $\Sigma$ be the apartment that corresponds to the abelian subgroup $A\cong \mathbb{Z}^{m}$ of the affine Weyl group $W^{\text{\tiny{aff}}}$ of $G$. As $A$ can be lifted to an abelian subgroup of $G$, one can choose a basis of hyperbolic elements $\{\gamma_1,\cdots, \gamma_m\} \in \Stab_G(\Sigma)$ for the lift of $A $ such that the attracting endpoint of $\gamma_j$, for every $j \in \{1, \cdots, m\}$, is a vertex in $\bd \Sigma$  (i.e., it is a vertex of the chamber at infinity that corresponds to the Borel subgroup $B$). It is then easy to see that for every simplex $\sigma \subset \partial \Sigma$ there exists a hyperbolic element in $\Stab_{G}(\Sigma)$ with attracting endpoint in the interior of $\sigma$. So $P^{I}, U^{I}, M_{I}$ are  well defined for any simplex $\sigma \subset \partial \Sigma$.
\end{remark}

\begin{remark}
Let $\mathbb{G}$ be a semi-simple algebraic group over a non Archimedean local field $k$ and let $G=\mathbb{G}(k)$ be the group of all $k$--rational points of $\mathbb{G}$. By the Bruhat--Tits theory and~\cite[Cor.~4.29]{Cio} the groups $U^{\pm}_{a}$ are closed.
\end{remark}

\begin{example}
\label{ex:levi_decom}
For the totally disconnected locally compact group $G=\SL(n, \QQ_p)$ the Levi decompositions from Proposition~\ref{coro::geom_levi_decom} are as follows. 

 Consider the subgroups of $\SL(n, \QQ_p)$:
$$\begin{array}{ccc} 
B & M_{\emptyset}&  U^{\emptyset}\\ 
\begin{pmatrix} \alpha _1 & * & \cdots & * \\
0 & \alpha_2 & * & \cdots \\
0 & 0 & \ddots & *\\
0& \cdots & 0 & \alpha_n
\end{pmatrix} &
 \begin{pmatrix} \alpha _1 & 0 & \cdots & 0 \\
0 & \alpha_2 & 0 & \cdots \\
0 & 0 & \ddots & 0\\
0& \cdots & 0 & \alpha_n
\end{pmatrix}&
 \begin{pmatrix} 1 & * & \cdots & * \\
0 & 1 & * & \cdots \\
0 & 0 & \ddots & *\\
0& \cdots & 0 & 1
\end{pmatrix}
\end{array}.$$

The subgroup $B$ is the Borel subgroup corresponding to the model ideal chamber $c \in \Ch(\Sigma)$.
Then  $B = G_{c}=P^{\emptyset}= U^{\emptyset} M_{\emptyset} $. Note the subgroup $U^{\emptyset}$ is closed in $\SL(n, \QQ_p)$ and it is the unipotent radical of $B$ as defined in Remark \ref{rem::uni_rad}; it is a fact that $U^{\emptyset}$ contains only elliptic elements. Also $M_{\emptyset}=C$ is contained in $\Stab_G(\Sigma)$. So $M_{\emptyset}$ stabilizes the apartment $\Sigma$; it contains both elliptic and hyperbolic elements. The subgroup $M_{\emptyset}^{0} \leq  \SL(n, \ZZ_p) $ contains all the elliptic elements of $M_{\emptyset}$.


More generally, we associate the following subgroups of $\SL(n, \QQ_p)$  with a Coxeter subsystem $(W_{I},I)$ of $(W,S)$:
$$\begin{array}{ccc} 
P^{I} & M_{I}&  U^{I}\\ 
\begin{pmatrix} A_1 & * & \cdots & * \\
0 & A_2 & * & \cdots \\
0 & 0 & \ddots & *\\
0& \cdots & 0 & A_k\\
\end{pmatrix}&
 \begin{pmatrix} A_1 & 0 & \cdots & 0 \\
0 & A_2 & 0 & \cdots \\
0 & 0 & \ddots & 0\\
0& \cdots & 0 & A_k\\
\end{pmatrix}&
\begin{pmatrix} Id_1 & * & \cdots & * \\
0 & Id_2 & * & \cdots \\
0 & 0 & \ddots & *\\
0& \cdots & 0 & Id_k\\
\end{pmatrix}
\end{array}$$
where the blocks $A_1, \cdots, A_k$ are \textbf{indecomposable,} square matrices of possibly different dimensions. (For B the blocks are one dimensional and $k=n$).  A block $A_i$ is \textbf{indecomposable} if no conjugate of $P^I$ in $SL_n (\QQ_p)$ allows us to write $A_i$ as a direct sum of smaller blocks. The dimensions of the blocks $A_1, \cdots, A_k$ are determined by $(W_{I},I)$.  We abuse notation and denote by $Id_1, \cdots, Id_k$ identity matrices of the same size as $A_1, \cdots, A_k$.

Then $B \leq P^{I}=U^{I} M_{I}$. The subgroup $U^{I}$ is the unipotent radical of $P^{I}$ as defined in Remark \ref{rem::uni_rad}, it is a closed subgroup of $\SL(n, \QQ_p)$ and $U^{I} \leq U^{\emptyset}$. The Levi factor $M_{I}$ is a reductive group \cite[Definition 11.22]{Borel} and it is unique up to conjugacy. Its elliptic and hyperbolic part and the ``flat'' that $M_{I}$ stabilizes are a  special case of the general case that will be presented below.
\end{example}

\begin{remark}  Recall, $\PSL(n, \QQ_p)$ acts faithfully on the Bruhat--Tits building $\Delta$ associated with $\SL(n, \QQ_p)$ and the kernel of the isogeny $\SL(n, \QQ_p)\rightarrow \PSL(n, \QQ_p)$ is precisely $\mu_n$, the group of $n$-th roots of unity in $\QQ_p^*$.  As $\mu_n \cdot \Id$ is a normal subgroup of $\SL(n, \QQ_p)$ and acts trivially on the building $\Delta$,  $\PSL(n, \QQ_p)$ and $\SL(n, \QQ_p)$ give the same action on $\Delta$. 
\end{remark}

\textbf{Further decomposition of $M_{I}$}

First we fix the following notation. Let $\Delta$ be a locally finite thick affine building of dimension $n$ with its complete system of apartments. Let $\Sigma$ be an apartment of $\Delta$, $\sigma_{+}$ and $\sigma_{-}$ be two opposite simplices in $\partial \Sigma$,  and let $(W_{I},I)$ be the corresponding Coxeter subsystem of $\sigma_{+}$. The \textbf{residue} $res(\sigma_{+})$ is the set of all ideal chambers in $\Ch(\partial \Delta)$ that contain the simplex $\sigma_{+}$. Let $\Delta(\sigma_{+}, \sigma_{-})$ be the union of all the apartments of $\Delta$ whose ideal boundaries contain $\sigma_{+}, \sigma_{-}$.


\begin{proposition}
\label{prop::res_building}
The union of apartments $\Delta(\sigma_{+}, \sigma_{-})$ is a closed convex subset of $\Delta$,  it is an extended locally finite thick  affine building and $\Delta(\sigma_{+}, \sigma_{-}) \cong \RR^{\vert I \vert} \times \Delta_{I}$ with $\Delta_{I}$ a locally finite thick affine building of dimension $n-\vert I \vert$. Moreover, $res(\sigma_{+})$ is a compact subset of $\Ch(\partial \Delta)$, it is a spherical building and $res(\sigma_{+}) \cong \Ch(\partial \Delta_{I})$. 
\end{proposition}

\begin{proof}
This is proved by Rousseau~\cite[4.3]{Rou11} in a more general setting (see also \cite[Section 2.6]{CMRH}). By \cite[5.30]{AB} we know  $res(\sigma_{+})$ is a spherical building. 
\end{proof}

When $\sigma_{+}$ and $\sigma_{-}$ in Proposition \ref{prop::res_building} are two opposite chambers in $\partial \Sigma$, then $\Delta(\sigma_{+}, \sigma_{-}) \cong \RR^{\vert n \vert} \cong \Sigma$ and $\Delta_{I}$ is just a point.

Note the Euclidean factor $ \RR^{\vert I \vert}$ from the splitting of $\Delta(\sigma_{+}, \sigma_{-})$ is in fact the affine space $L_{I}$ in Guivarc'h--R\'emy~\cite[1.2.1]{GR}. The simplices $\sigma_{+}, \sigma_{-}$ are in the boundary at infinity of $\RR^{\vert I \vert}$.

\begin{proposition}
\label{prop::res_building_action}
Let $G$ be a closed, type-preserving subgroup of $\Aut(\Delta)$ with a strongly transitive action on $\Delta$. Then $G_{\sigma_{+}, \sigma_{-}}:=G_{\sigma_+} \cap G_{\sigma_{-}} =M_I$  acts strongly transitively on $\partial \Delta_I$.
\end{proposition}

\begin{proof}
Since $G$ acts strongly transitively on $\Delta$ then $G$ acts strongly transitively on $\partial \Delta$. 
By Proposition \ref{prop::res_building} every apartment of $\partial \Delta_I$ is the boundary at infinity of some apartment (not necessarily unique) in $\Delta(\sigma_{+}, \sigma_{-})$.

Let $c_1,c_2 \in res(\sigma_{+})$ and let $A_1, A_2$ be two apartments of $\Delta(\sigma_{+},\sigma_{-})$ with $c_i \in \Ch(\partial A_i)$, for $i \in \{1,2\}$. Then by the strong transitivity of $G$ there exists $g \in G$ with $g(A_1)=A_2$ and $g(c_1)=c_2$. In particular, $g(\Ch(\partial A_1))=\Ch(\partial A_2)$. Since $c_1$ and $c_2$ share the same simplex $\sigma_+$ and $G$ is type-preserving, then $g(\sigma_{+})=\sigma_+$. So $g(\sigma_{-})=\sigma_{-}$ because $A_1, A_2$ are apartments of $\Delta(\sigma_{+},\sigma_{-})$. Thus $g \in M_I$. 
\end{proof}

\begin{proposition}
\label{prop::res_building_str_tran}
Let $G$ be a closed, type-preserving subgroup of $\Aut(\Delta)$ with a strongly transitive action on $\Delta$. Then $M_{I}$ acts on $\Delta(\sigma_{+}, \sigma_{-})$ and preserves the splitting $\Delta(\sigma_{+}, \sigma_{-}) \cong \RR^{\vert I \vert} \times \Delta_{I}$. 
Moreover, the induced map $\alpha : M_{I} \to \Isom(\Delta(\sigma_{+}, \sigma_{-}))=\Isom(\RR^{\vert I \vert}) \times \Isom(\Delta_{I}) $ is a group homomorphism and the normal subgroups $H_I := \alpha^{-1}( \alpha\vert_{\Isom(\Delta_{I})}(M_{I}))$, resp.,  $H^I := \alpha^{-1}( \alpha\vert_{\Isom(\RR^{\vert I \vert})} (M_{I}))$ of $M_{I}$  act by automorphisms and strongly transitively on $\Delta_{I}$, resp., by translations on $\RR^{\vert I \vert}$.
\end{proposition}


\begin{proof}
As every $g \in M_I$ stabilizes $\sigma_+$ and $\sigma_-$ and is type-preserving, then $g$ preserves the splitting $\RR^{\vert I \vert} \times \Delta_{I}$ and acts by translations (or as a trivial element) on $\RR^{\vert I \vert}$ and by  automorphisms on $\Delta_{I}$. 

By the splitting results of \cite[Theorem 1.1, Corollary 1.3]{FL} we have $\Isom(\Delta(\sigma_{+}, \sigma_{-}))= \Isom(\RR^{\vert I \vert}) \times \Isom(\Delta_{I})$. As the action of $M_{I}$ on $\Delta(\sigma_{+}, \sigma_{-})$ is a group action, the corresponding map $\alpha : M_{I} \to \Isom(\RR^{\vert I \vert}) \times \Isom(\Delta_{I}) $  is a group homomorphism. Since $\Isom(\RR^{\vert I \vert}), \Isom(\Delta_{I})$ are normal subgroups in $\Isom(\Delta(\sigma_{+}, \sigma_{-}))$ so are the subgroups $H_I ,H^{I}$ in $M_{I}$.

By construction, $H^{I}$ is a subgroup of translations of $\RR^{\vert I \vert}$. It remains to prove $H_I$ acts strongly transitively on $\Delta_I$. Indeed, first note $H^I$ does not affect the action on $\Delta_I$ and $res(\sigma_{+})\cong \Ch(\partial \Delta_{I})$. So, as $M_I$ acts strongly transitively on $res(\sigma_{+})$, we obtain $H_I$ acts strongly transitively on $\partial \Delta_I$. By \cite[Theorem B]{KS} $H_I$ acts strongly transitively on $\Delta_I$ and the conclusion follows.
\end{proof}

\begin{remark}
\label{rem::alpha_map}
In general the map $\alpha : M_{I} \to \Isom(\RR^{\vert I \vert}) \times \Isom(\Delta_{I})$ from Proposition \ref{prop::res_building_str_tran} is not injective as the subgroup $Ker(\alpha) \leq M_{I}$ that acts trivially on $\Delta(\sigma_{+}, \sigma_{-}) \cong \RR^{\vert I \vert} \times \Delta_{I}$ can be non-trivial.
In addition, the image subgroup $\alpha(M_I)$ might not split as a direct product in  $\Isom(\RR^{\vert I \vert}) \times \Isom(\Delta_{I})$. 
Both issues occur in the special case $ M_{I} \leq  G= \SL(n, \QQ_p)$.
\end{remark}

We have seen in Example \ref{ex:levi_decom} that the Levi factor $M_I$ of a parabolic subgroup $P^{I}$ of $\SL(n, \QQ_p)$,  is a reductive group 
$$M_{I} =\{\begin{pmatrix} A_1 & 0 & \cdots & 0 \\
0 & A_2 & 0 & \cdots \\
0 & 0 & \ddots & 0\\
0& \cdots & 0 & A_k\\
\end{pmatrix} \; \vert \; \det(A_1) \cdot \det(A_2) \cdot ... \cdot \det(A_k) =1\},$$
where the blocks $A_1, \cdots, A_k$ are indecomposable, square matrices of possibly different dimensions. The next lemma computes the normal subgroups $Ker(\alpha)$ and $H^{I}$ for the map $\alpha : M_{I} \to \Isom(\Delta(\sigma_{+}, \sigma_{-}))=\Isom(\RR^{\vert I \vert}) \times \Isom(\Delta_{I}) $ from Proposition \ref{prop::res_building_str_tran}.

\begin{lemma}
\label{rem::SL_Levi_factor}
We have 
$$Ker(\alpha) = \{\begin{pmatrix} \lambda_1 \Id_1 & 0 & \cdots & 0 \\
0 & \lambda_2 \Id_2 & 0 & \cdots \\
0 & 0 & \ddots & 0\\
0& \cdots & 0 & \lambda_k \Id_k\\
\end{pmatrix} \in \SL(n, \QQ_p)  \; \vert \; \lambda_1, \lambda_2, ..., \lambda_k \in \ZZ_p\}$$

$$H^{I}=\{\begin{pmatrix} \lambda_1 \cdot \Id_1 & 0 & \cdots & 0 \\
0 & \lambda_2 \cdot \Id_2 & 0 & \cdots \\
0 & 0 & \ddots & 0\\
0& \cdots & 0 & \lambda_k \cdot \Id_k \\
\end{pmatrix} \in \SL(n, \QQ_p) \; \vert \;  \lambda_1, \lambda_2, ..., \lambda_k \in \QQ_p\}, $$
where $\Id_1, \Id_2, \cdots, \Id_k$ are square identity matrices of possibly different dimensions, but respectively, with the same size as $A_1, A_2, \cdots, A_k$  from $M_{I}$ above.
\end{lemma}
\begin{proof}

Notice every element in $Ker(\alpha)$ fixes pointwise $\Delta(\sigma_{+}, \sigma_{-})$, and so $Ker(\alpha)$ fixes pointwise the apartment $\Sigma$ corresponding to the Cartan subgroup $C \leq \SL(n, \QQ_p)$. Thus $Ker(\alpha) \leq \SL(n, \ZZ_p) \cap B \cap \Stab_{G}(\Sigma)= Diag(n, \ZZ_p)$, the diagonal matrices in $\SL(n, \ZZ_p)$. Recall $\Stab_{G}(\Sigma)$ is the monomial subgroup of $\SL(n, \QQ_p)$ (see \cite[Section 6.9, page 351]{AB}). But $Ker(\alpha)$ is normal in $M_{I}$ and a subgroup in $Diag(n, \ZZ_p)$, thus it is easy to see $Ker(\alpha)$ has the desired form.

Now, let us compute $H^{I}$. Indeed, we first claim  $H^{I}$ is a subgroup of $Diag(n, \QQ_p)$. This is true because every element of $H^{I}$ acts trivially on $\Delta_{I}$ and translates (or acts trivially) on the flat $\RR^{\vert I \vert}$. Recall $\RR^{\vert I \vert}$ is a subflat of  $\Sigma$ and so every element of $ H^{I}$ stabilises the apartment $\Sigma$ and fixes pointwise the boundary of $\Sigma$. Thus $H^{I}$ is a subgroup of $\Stab_{G}(\Sigma) \cap B= Diag(n, \QQ_p)$ and our claim follows. As $H^{I}$ is also a normal subgroup of $M_I$ we obtain the desired conclusion.
\end{proof}

\begin{remark}
\label{rem::M_I_elements}
Notice, the intersection $H_I \cap H^I$ is $Ker(\alpha)$ consisting of diagonal matrices with entries in $\ZZ_p$, and those elements are not hyperbolic. Therefore hyperbolic elements of $H_I$ are not hyperbolic elements in $H^I$. 
\end{remark}
 
As $H_I$ acts strongly transitively on $\Delta_I$ we can apply all the decompositions presented above. Let $y \in \Delta_I$, then $\Fix_{M_I}(\RR^{\vert I \vert} \times y):=\{ g \in M_I \; \vert \; g(z)=z, \forall z \in \RR^{\vert I \vert} \times y \}= (H_I)_{y} $. 

Following the definition in Guivarc'h--R\'emy~\cite[1.2.3]{GR} we have:

\begin{definition}
\label{def::K_x_I}
For $y \in \Delta_I$ we denote $K_{I,y}:=\Fix_{M_I}(\RR^{\vert I \vert} \times y)= (H_I)_{y}$, $D_{I,y}:= U^{I} K_{I,y}$ and $R_{I,y}:= U^{I} H^{I} K_{I,y}.$
\end{definition}

\section{Chabauty limits of parahoric subgroups}
\label{Chabauty_parahoric}

This section gives a different geometric proof of~\cite[Theorem 3, Corollary 4]{GR} and~\cite[Theorem 3.14]{Htt} for $\SL(n, \QQ_p)$ by using Lemma~\ref{lem::same_trans_length} and the decomposition of $M_I$ given in Proposition~\ref{prop::res_building_str_tran}. The proof of \cite[Theorem 3]{GR} uses a probabilistic method which holds for general semi-simple algebraic groups $G$ over non Archimedean local fields. For completeness we also give more geometric proofs of two main Lemmas in~\cite{GR,Htt} about general semi-simple algebraic groups $G$ over non Archimedean local fields.

Let $\mathbb{G}$ be a semi-simple algebraic group over a non Archimedean local field $k$ and let $G=\mathbb{G}(k)$ be the group of all $k$--rational points of $\mathbb{G}$. In particular, we will consider $G=\SL(n, \QQ_p)$.  Denote by $\Delta$ the corresponding locally finite thick affine building of dimension $n-1$ on which $G$ acts by type-preserving automorphisms and strongly transitively.  Let $\Sigma$ be the apartment of $\Delta$ and fix a special vertex $x$ in $\Sigma$, the base point. For a point $y \in \Delta$ we denote $K_{y}:=\Stab_{G}(y)=:G_y$; note $K_y$ is a closed subgroup of $G$. 

Given a sequence $\{x_l\}_{l \in \NN} \subset \Delta$ of points we want to study the Chabauty limits in $\mathcal{S}(G)$ of the sequence of closed subgroups $\{K_{x_l}\}_{l \in \NN}$. By Guivarc'h--R\'emy~\cite[Lemma~1,~1.3.3]{GR} for every sequence $\{x_l\}_{l \in \NN} \subset \Delta$ there exist $I \subset S$, a converging sequence $\{g_{l_k}\}_{k \in \NN} \subset K_{x}$ and a simplex $\sigma \subset \Delta_I$ such that  for every $k \in \NN$ we have: $x_{l_k}=g_{l_k}(y_{l_k})$,  $y_{l_k} $ is in the Weyl chamber with base point $x$ of $\Delta$ and the projection of $y_{l_k}$ to $\Delta_I$ is in $\sigma$ (comprising the boundary of $\sigma$). The sequence $\{y_{l_k}\}_{k \in \NN}$ is called \textbf{$I$--canonical}.


\medskip
The following two lemmas hold true for any $I$-canonical sequence $\{y_l\}_{l \in \NN} \subset \Delta$ with $\sigma$ its associated simplex in $\Delta_I$ and $y$ a point in the interior of $\sigma$.

\begin{lemma}(See \cite[Lemma 7]{GR})
\label{lem::D_in_D_I}
Suppose the sequence $\{K_{y_l}\}_{l \in \NN}$ converges to $D \in \mathcal{S}(G)$ with respect to the Chabauty topology on $\mathcal{S}(G)$. Then $D_{I,y} \subset D$.
\end{lemma}

\begin{proof}
We use the Levi decomposition. Since $D_{I,y}= U^{I} K_{I,y}$ we prove the lemma in two steps: the unipotent radical $U^{I}$ is a subset of $D$ and $K_{I,y} \subset D$. 

To prove the inclusion $U^{I} \subset D$ it is enough by~\cite[Cor.~4.29]{Cio} to prove every root group $U_{h}$ of $U^{I}$ is in $D$. 
Fix a root group $U_{h}$ as in~\cite[Cor.~4.29]{Cio}  and some $u \in U_{h}$. By~\cite[Cor.~4.24]{Cio} $u$ fixes pointwise a half-apartment of $\Delta$ that contains in its interior the Weyl chamber, but a finite volume part. Then there exists $N_u >0$ such that for every $l \geq N_u$ we have $u(y_l)=y_l$, thus $u \in K_{y_l}$ and so $u \in D$. 
Now we show $D$ contains $K_{I,y}$.  Indeed, as $y$ is in the interior of $\sigma$ and by the definition of $K_{I,y}$ we have  $K_{I,y}=K_{I,\sigma}$ and for $l$ big enough $K_{I,y}$ fixes pointwise $y_l$, so $K_{I,y} \subset K_{y_{l}}$. 
\end{proof}

\begin{lemma}(See~\cite[Lemma 3.12]{Htt})
\label{lem::D_in_P_I}
Suppose the sequence $\{K_{y_l}\}_{l \in \NN}$ converges to $D \in \mathcal{S}(G)$ with respect to the Chabauty topology on $\mathcal{S}(G)$. Then $D \subset P^{I}$.
\end{lemma}

\begin{proof}
The proof goes as for~\cite[Lemma 3.12]{Htt}. 
By hypothesis we know the sequence $\{y_l\}_{l \in \NN}$ is contained in the closed strip $\RR^{\vert I\vert} \times \sigma$ intersected with the Weyl chamber in $\Delta$ with base point $x$. Therefore $\{y_l\}_{l \in \NN}$  admits a converging subsequence to a point $\xi \in \partial \Delta$ whose stabilizer is $P^{I}$. Let $d \in D$ and take $\{k_l \in K_{y_l}\}_{l \in \NN}$ that converges to $d$. Then we have, restricting eventually to a subsequence, $\{k_l(y_l)=y_l\}_{l \in \NN}$ converges to $d(\xi)$ and also to $\xi$; thus $d(\xi)=\xi$ and so $d \in P^{I}$. 
\end{proof}

\begin{theorem}(See~\cite[Theorem 3, Corollary 4]{GR} and~\cite[Theorem 3.14]{Htt})
\label{thm::main_thm}
Let $G=\SL(n, \QQ_p)$ with its associated Bruhat--Tits building $\Delta$, and $x$ be a fixed special vertex of $\Delta$.
Let $\{x_l\}_{l \in \NN} \subset \Delta$. Then the sequence $\{K_{x_l}\}_{i \in \NN}$ admits a convergent subsequence with respect to the Chabauty topology on $ \mathcal{S}(G)$ and the corresponding limit is $K_x$-conjugate to some $D_{I,y}$, for some $I \subset S$ and $y \in \Delta_{I}$.
\end{theorem}

\begin{proof}
By Guivarc'h--R\'emy~\cite[Lemma~1,~1.3.3]{GR}, it is enough to consider an $I$-canonical sequence $\{y_l\}_{l \in \NN} \subset \Delta$ such that $\{K_{y_l}\}_{l \in \NN}$ converges to $D \in \mathcal{S}(G)$. Let $\sigma$ be the associated simplex in $\Delta_I$ of $\{y_l\}_{l \in \NN} $ and let $y$ be inside $\sigma$.  By the above lemmas we have $D_{I,y} \subset D \subset P^{I}$. Let $d \in D$.
By Lemma~\ref{lem::D_in_D_I} $U^{I} \subset D$ and we can assume that $d \in M_{I} $. By Proposition \ref{prop::res_building_str_tran} $d$ cannot project to a non-trivial translation in $H^{I}$, as this contradicts the fact that $d$ is a limit of elliptic elements (see Lemma \ref{lem::same_trans_length}). Therefore, we have $d \in H_I$. 

As every vertex in the affine building of $\Delta$ of $\SL(n, \QQ_p)$ is special, every vertex of the affine building $\Delta_I$ of $H_I$ is also special. So for every vertex $z \in \Sigma_I$ the maximal compact subgroup $(H_I)_z$ is good  and $H_I$ admits the Iwasawa decomposition $H_I= (H_I)_z A_I B_I^{0}$, where $A_I$ is the maximal split torus of $H_I$ with respect to the apartment $\Sigma_I$ and $B^{0}_{I}$ is the pointwise stabilizer in $H_I$ of a ideal chamber in $\partial \Sigma_I$ corresponding to the Weyl chamber in $\Sigma$. 

Suppose that $\sigma$ is a vertex, so $y=\sigma$. As in this case $(H_I)_{\sigma}= K_{I,y}$ and $ K_{I,y}$ is contained in $D$ by Lemma~\ref{lem::D_in_D_I}, we can assume  (up to multiplying by the inverse of an element in $K_{I,y}$) that $d \in A_I B_I^{0}$. Again by Lemma~\ref{lem::same_trans_length} $d \in B_I^{0}$, as otherwise $d$ would translate along some axis in the building $\Delta_I$ and thus in $\Delta$ contradicting the fact that $d$ is a limit of elliptic elements. So $d$ fixes pointwise an ideal chamber in $\partial \Sigma_I$ and so the ideal chamber in $\partial \Sigma$ corresponding to the Weyl  chamber of $\Delta$. This implies $d$ fixes pointwise an entire Weyl subsector contained in $\Sigma$.  Since $d \in M_I$ then $d$ preserves the splitting $\RR^{\vert I \vert} \times \Delta_{I}$. Since it is elliptic, $d$ fixes pointwise a flat in $\Delta$ of dimension $\vert I \vert$ that is parallel to the flat $\RR^{\vert I \vert}$ corresponding to $M_I$. In particular, by~\cite[Chapter II, Prop. 6.9]{BH99} $\Min(d)= \RR^{\vert I \vert} \times Y$, for some $Y \subset \Delta_{I}$. As $\Min(d)$ is closed, convex and contains a flat of dimension $\vert I \vert$ and a Weyl subsector,  $d$ must fix all points $y_l$, for $l$ big enough.  By projecting to $\Sigma_I$ we have $d \in K_{I,y}= (H_I)_{\sigma}$. We have obtained $D \subset D_{I,\sigma}$ if $\sigma$ is a vertex.

Suppose now $\sigma$ is not a vertex. As in Guivarc'h--R\'emy~\cite[proof of Lemma~8]{GR} $D_{I,y} = D_{I,\sigma} =\bigcap\limits_{z \text{ vertex of } \sigma} D_{I,z}$ and $D \subset D_{I,z}$ for every $z$ vertex of $\sigma$ and the conclusion follows.
\end{proof}

Recall from Remark~\ref{rem::special_vertex} the Iwasawa decomposition does not hold for a general non-special vertex, and so the proof of Theorem~\ref{thm::main_thm} cannot be generalised to semi-simple algebraic group over a non Archimedean local field.

\section{Chabauty limits of diagonal Cartan subgroups and Classification}

In this section we study Chabauty limits of a different family of closed subgroups: the set of all $\SL(n, \QQ_p)$-conjugate of the Cartan subgroup $C$ of $G:=\SL(n, \QQ_p)$ as defined in Section~\ref{sec::sec_three}.  The diagonal subgroup of $G$ is called the \textbf{diagonal Cartan subgroup}, $C$. It is a closed subgroup of $G$, and $C \leq \Stab_G(\Sigma)$, for a unique apartment $\Sigma$ in the Bruhat--Tits building $X$ associated with $G=\SL(n,\QQ_p)$. Set $$Cart(G) := \{ g C g^{-1} : g \in G\} \subset \mathcal{S}(G)$$ the set of all conjugates in $G$ of $C$.   
Let $\overline{Cart(G)}^{Ch}$ be the  closure of $Cart(G)$ in $\mathcal{S}(G)$ with respect to the Chabauty topology on $\mathcal{S}(G)$;  so $\overline{Cart(G)}^{Ch}$ is compact.  An element in $\overline{Cart(G)}^{Ch}$ is called a \textbf{Chabauty limit of the diagonal Cartan}.  For simplicity, we will use \textbf{limit of the Cartan} as terminology.


\begin{lemma} 
\label{lem::fix_cartan}
The diagonal Cartan subgroup $C \leq \Stab_G(\Sigma)$ fixes pointwise the boundary at infinity $\partial \Sigma$ of the apartment $\Sigma $ of $X$.

\end{lemma}

\begin{proof}
Since $C$ is abelian and acts cocompactly on $\Sigma$, the restriction of every element $\eta$ of $C$ to $\Sigma$ is either hyperbolic or pointwise fixes $\Sigma$. In particular, we have $\Sigma \subset \Min(\eta)$, for every $\eta \in C$. In both cases ($ \eta$ is hyperbolic or elliptic) the boundary $\partial \Sigma$ is pointwise fixed by every $\eta \in C$, and the conclusion follows.
\end{proof}

\begin{proposition}
\label{prop::limit_cartan_in_borel}
 Any limit of the diagonal Cartan  is contained in the product of the group $\mu_n$ of $n$-roots of 1  in $\QQ_p$, with a conjugate of the Borel subgroup $B$ and fixes pointwise the closure of an ideal chamber of $\Ch(\partial X)$ that is not necessarily unique. (Here $\mu_n$ is identified with the corresponding group of scalar matrices).
\end{proposition}

\begin{proof}
Note that $\mu_n$ appears as the kernel of the quotient map $\SL_n(\QQ_p)\rightarrow \PSL_n(\QQ_p)$: indeed it is only $\PSL_n(\QQ_p)$ that acts faithfully on the corresponding Bruhat-Tits building $X$.

Let $\{g_m\}_{m \in \NN} \subset G$ and suppose the limit $\lim\limits_{m \to  \infty} g_mCg_m^{-1}$ exists in $\mathcal{S}(G)$ and equals $H$. 
Let $\Sigma$ be the apartment of $X$ such that $C \subset \Stab_{G}(\Sigma)$. In particular, every ideal chamber $c \in \Ch(\partial \Sigma)$ is fixed by $C$. Take one such ideal chamber $c \in \Ch(\partial \Sigma)$. Recall $X \cup \partial X$ is a Hausdorff compact space with respect to the cone topology on $X \cup \partial X$. Since the action of $G$ on $X \cup \partial X$ is continuous, the sequence $\{ g_m (c)\}_{m \in \NN}$ admits a convergent subsequence $g_{m_i}(c) \to c' \in \Ch(\partial X)$, when $m_i \to \infty$. For simplicity, we denote that subsequence also by $\{g_m\}_{m \in \NN}$. 

Let $h \in H$, there exists a sequence $\{h_m \in g_mCg_m^{-1} \}_{m \in \NN}$ such that $h_m \to h$, when $m \to \infty$, with respect to the  induced subgroup topology on $H$ from $G$. The action of $G$ on $X \cup \partial X$ is continuous thus so is the action of $H$, 
so we have $h_m g_m(c)\to h(c')$, when $m \to \infty$. But $h_m g_m(c)=g_m(c) \to c'$. Thus $h(c')=c'$, for every $h \in H$. Therefore $H$ is a subgroup of $\Stab_G(c')$ that is contained in a conjugate of $\mu_n\cdot B$. As $G$ acts by type-preserving automorphisms, the ideal chamber $c'$ and its closure are pointwise fixed by $H$.

Note for every ideal chamber $c \in \Ch(\partial \Sigma)$ there is an ideal chamber $c' \in \Ch(\partial X)$, such that $H \leq \Stab_G(c')$, but maybe for $c_1 \neq c_2 \in \Ch(\partial \Sigma)$ the corresponding $c'_1, c'_2$ are different. This means there could be more ideal chambers stabilized by $H$. 
\end{proof}

Examples of limits of the diagonal Cartan that fix pointwise more ideal chambers of $\Ch(\partial X)$ are all conjugates of $C$. Any conjugate of $C$ pointwise fixes the ideal boundary of an apartment of $X$.

\medskip
Next we show every limit of $C$ is conjugate to a limit of $C$ under conjugating by a sequence in the unipotent radical $U$ of $B$. 

\begin{lemma}
\label{lem::unip_conj}
Let $H$ be a  limit of the diagonal Cartan. Then there exists a sequence $\{u_m\}_{m \in \NN} \subset U$, where $U$ is the unipotent radical of the Borel subgroup $B$, and some $k \in K$ such that $\{u_m C u_m^{-1}\}_{m \in \NN}$ converges to a limit  $H'$ of the diagonal Cartan  and $kH'k^{-1}=H$.
\end{lemma}

\begin{proof}
As $H$ is a limit of the diagonal Cartan there exists $\{g_m\}_{m \in \NN} \subset \SL(n, \QQ_p)$ such that  $\{g_m C g_m^{-1}\}_{m \in \NN}$ converges to $H$ in the Chabauty topology. 

As $\SL(n, \QQ_p)= KB=KUC$, where $C$ is the diagonal Cartan subgroup of all diagonal matrices, we have $g_m=k_m u_m a_m$, with $k_m \in K, u_m \in U$ and $a_m \in C$. Because $K$ is compact, the sequence $\{k_m\}_{m \in \NN}$ admits a convergent subsequence to some $k \in K$. Abusing notation we write $k_m \xrightarrow{m \to \infty}  k$.

Then $\{u_m a_m C a_m^{-1} u_m^{-1}=u_m  C u_m^{-1} \}_{m \in \NN}$ admits a convergent subsequence in $\mathcal{S}(\SL(n, \QQ_p))$. Abusing notation again we write $u_m  C u_m^{-1} \xrightarrow{m \to \infty}  H'$. So $H'$ is a limit of the diagonal Cartan.

We claim $kH'k^{-1}=H$. Indeed, it is easy to see  $kH'k^{-1} \leq H$. Take now $h \in H$, then there exists $\{h_m \in C\}_{m \in \NN }$ such that $k_m u_m h_m u_m^{-1} k_m^{-1} \xrightarrow{m \to \infty}  h$. But $k_m^{-1}h k_m  \xrightarrow{m \to \infty}  k^{-1}hk$, so $u_m h_m u_m^{-1} = k_m^{-1} k_m u_m h_m u_m^{-1} k_m^{-1} k_m \xrightarrow{m \to \infty} k^{-1}hk$. Thus $k^{-1}hk \in H'$ and the conclusion follows.
\end{proof}

\begin{remark}
To classify  limits of $C$ up to conjugacy, it suffices by Proposition~\ref{prop::limit_cartan_in_borel} to find all limits of $C$ that are subgroups of the Borel, $B$. Note a limit $H$ of the diagonal Cartan either contains only elliptic elements and in which case we call $H$ an \textbf{elliptic limit of the Cartan} otherwise $H$ contains at least one hyperbolic element and we call $H$ a \textbf{hyperbolic limit of the Cartan}. 
\end{remark}


\subsection{A homeomorphism between $\overline{Cart (\mathfrak{sl}(n, \QQ_p))}^{Ch}$ and $\overline{Cart (\SL(n, \QQ_p))}^{Ch}$}\label{gr_alg}
\label{homeo_Gr} 

By Bourbaki \cite[Thm.~2,~pg.~340]{Bou75} we know that every closed subgroup of a real or $p$-adic Lie group is again a real or $p$-adic Lie group (for the definition of a real or $p$-adic Lie group see Bourbaki \cite{Bou75}). In particular, every closed subgroup of $\SL(n, \QQ_p)$ is a $p$-adic Lie group.

Again by  Bourbaki \cite[Def. 6, pg. 252]{Bou75} every real or p-adic Lie group $H$ has an associated real or $p$-adic Lie algebra $(\mathfrak{h}, [,])$. This gives a Lie functor $L$ from the set of real or p-adic Lie groups to the set of real or p-adic Lie algebras. In particular, one can consider the restriction of the Lie functor $L$ to the set of all closed subgroups of a Lie group $H$ to the set of all closed Lie subalgebras of $(\mathfrak{h}, [,])$. In the $p$-adic case the Lie functor $L$ is not injective (see Bourbaki \cite[Thm. 3, pg. 283]{Bou75}): $L(H_1)=L(H_2)$ implies $H_1 \cap H_2$ is open in $H_1$ and $H_2$. So the Lie algebra does not uniquely determine the Lie group for the $p$-adic case.  It would be interesting to see when  the Lie functor (real or $p$-adic case) $L: \mathcal{S}(H) \to \mathcal{S}(\mathfrak{h},[,])$ is continuous with respect to the Chabauty topology on $ \mathcal{S}(H)$ and the Chabauty topology on $\mathcal{S}(\mathfrak{h},[,])$, or on which compact subsets of $\mathcal{S}(H)$ the Lie functor $L$  is continuous. 

As in the case of real Lie groups, every $p$-adic Lie group $H$ has an associated exponential map $\exp: \mathfrak{h} \to H$ (see Bourbaki [Thm.~4, pg.~284, Def.~1, pg.~285]\cite{Bou75} or Hooke \cite{Hoo42}, Serre \cite{Ser92}, Schneider \cite{Sch11}). Still the  $\exp$ map for $p$-adic Lie algebras $\mathfrak h$ (see \cite[Examples (2), pg. 285]{Bou75}) is well defined only in a neighbourhood $V$ of zero of $(\mathfrak{h},[,])$ and maps $V$ diffeomorphically onto a neighbourhood of identity element in $H$ (for $\RR$ the exponential map is well defined on all of the Lie algebra).

In the case of $\SL(n, \QQ_p)$ and explicitly for the set $\overline{Cart (\SL(n, \QQ_p))}^{Ch}$ we propose a map that looks like an `inverse map' of the Lie functor $L$: $$Gr: \overline{Cart (\mathfrak{sl}(n, \QQ_p))}^{Ch} \to \overline{Cart (\SL(n, \QQ_p))}^{Ch}, \; \;  A \mapsto Gr(A):=\SL(n, \QQ_p) \cap \langle A, \Id \rangle. $$ 

 Here $Cart(\mathfrak{sl}(n, \QQ_p))$ is the set of all $\SL(n, \QQ_p)$-conjugates of the Lie subalgebra $\mathfrak{c} \subset \mathfrak{sl}(n, \QQ_p)$ of $C$. For $A \subset \mathfrak{sl}(n, \QQ_p)$ we denote $\langle A, \Id \rangle: = \{ \lambda \cdot a + \beta \Id : \lambda, \beta \in \QQ_p \textrm{ and } a \in A \}$ to be the $\QQ_p$-linear span of $A$ and $\Id$.

We show that $ A \in \overline{Cart (\mathfrak{sl}(n, \QQ_p))}^{Ch}$  is the Lie algebra of $Gr(A)$ and moreover the map $Gr$ is continuous with respect to the Chabauty topology. 


\subsubsection{\textbf{On closed subalgebras of $\mathcal{M}(n, \QQ_p)$ and $\mathfrak{sl}(n, \QQ_p)$}} In this subsection we build background on closed subgroups and subalgebras, before defining  the map $Gr$ in the next section.

All the topologies that we consider here, for example on $\mathcal{M}(n, \QQ_p)$, $\GL(n,\QQ_p)$, $\SL(n,\QQ_p)$, $\mathfrak{sl}(n, \QQ_p)$ etc.,  are the subspace topologies inherited from the product topology on $\QQ_p ^{n^2}$.

The set $\mathcal{M}(n, \QQ_p)$ of all $n \times n$ matrices over $\QQ_p$ is an algebra: it is a $\QQ_p$-vector space of finite dimension and the  multiplication of matrices gives the structure of an algebra.  

Set $G: = \SL(n,\QQ_p) \subset \GL(n , \QQ_p)$ and $G$ is a closed subgroup of $\GL(n , \QQ_p)$.  By Milne~\cite[Example~3.9, page~122]{Mil} the Lie algebra of $G$ is $\mathfrak{g}: = \mathfrak{sl}(n, \QQ_p)$, the set of all matrices in $\mathcal{M}(n, \QQ_p)$ of trace zero.  Note $\mathfrak{g}$ is a finite dimensional $\QQ_p$-vector subspace of $\mathcal{M}(n, \QQ_p)$ that is \underline{not} a subalgebra of $\mathcal{M}(n, \QQ_p)$ with respect to the usual multiplication of matrices. The multiplication on $\mathfrak{g}$ is given by the Lie bracket $[a, b]:= ab-ba$ for every $a,b \in \mathfrak{g}$. Note $[a,b]$ is of trace zero, for every $a,b \in \mathfrak{g}$, thus $[a,b] \in \mathfrak{g}$, and the identity matrix $\Id$ is \underline{not} in $\mathfrak{g}$.   Abelian subalgebra or just subalgebras of $\mathfrak{g}$ are considered with respect to the Lie bracket $[\cdot, \cdot]$.

The diagonal subgroup $Diag(n, \QQ_p) \subset \GL(n, \QQ_p)$ is a maximal abelian subgroup in $\GL(n, \QQ_p)$ and it is easy to see $Diag(n, \QQ_p)$ is a closed subgroup of $\GL(n, \QQ_p)$.  
We denote by $diag(n, \QQ_p)$ the set of all diagonal matrices of $\mathcal{M}(n, \QQ_p)$: it is a maximal abelian subalgebra of $\mathcal{M}(n, \QQ_p)$.

We denote by $\mathfrak{c}$ the \textbf{diagonal Cartan subalgebra} of the diagonal Cartan subgroup $C$.
We have $\mathfrak{c} : = diag(n, \QQ_p )\cap \mathfrak{g}$ and $\mathfrak{c}$ is an abelian subalgebra of $\mathfrak{g}$ with respect to the Lie bracket, as for every $a,b \in \mathfrak{c}$ we have $[a,b]=0$. So $\mathfrak{c}$ consists of all diagonal matrices of $\mathcal{M}(n, \QQ_p)$ with trace zero. We have $$Cart(\mathfrak{g}) = \{ g (diag(n, \QQ_p) \cap \mathfrak{g}) g^{-1} : g \in G\} = \{ g (diag(n, \QQ_p)) g^{-1} \cap \mathfrak{g} : g \in G\}.$$

\begin{definition}  Let $H \leq G$ be a subgroup. Denote by $\langle H \rangle$ the $\QQ_p$-linear span of elements of $H$. 
\end{definition} 

\begin{lemma}\label{cid} $\langle C\rangle = diag(n, \QQ_p )$.
\end{lemma} 

\begin{proof} 
For $a\in\QQ_p^*$, let $D_a$ be the diagonal matrix with diagonal entries $a,a^{-1},1,...,1$. So $D_a - D_1$ has diagonal entries $a-1,a^{-1}-1,0,...,0$. Choose distinct $a,b\in\QQ_p^*$, with $a\neq 1\neq b$, and observe that the matrices $D_a-D_1, D_b-D_1$ are linearly independent. So the linear span of $D_a, D_b, D_1$ contains the diagonal matrices with diagonal entries $1,0,0,...,0$ and $0,1,0,...,0$. Moving $a$ and $a^{-1}$ down the diagonal we construct all other elements of the canonical basis of $diag(n, \QQ_p )$, hence the result. Note that this proof holds over any field of order at least 4.
\end{proof} 

Observe $G$ acts on $Cart(G)$ and $Cart(\mathfrak{g})$ by conjugation. 

\begin{lemma}\label{stab} $Stab_G(C) := \{ g C g^{-1} = C : g \in G\}$ equals $Stab_G(\mathfrak{c}) := \{ g \mathfrak{c} g^{-1} = \mathfrak{c} : g \in G\}$ . 
\end{lemma}

\begin{proof} 

If $g\in G$ stabilizes $C$, then it stabilizes $\langle C\rangle=diag(n,\QQ_p)$ (by Lemma \ref{cid}), so it stabilizes the set of trace 0 matrices in $diag(n, \QQ_p )$, which is $\mathfrak{c}$. Conversely, if $g\in G$ stabilizes $\mathfrak{c}$, then it stabilizes $\langle\mathfrak{c},Id\rangle=diag(n, \QQ_p )$, hence it stabilizes $diag(n, \QQ_p )\cap G=C$.
\end{proof} 

The next remark is very useful for what follows.
\begin{remark}
\label{rem::Bou_closed}
 By Bourbaki \cite[Section I.14, Cor. 1]{Bou81}, in a Banach space over $\QQ_p$ every finite dimensional linear subspace is closed, similarly to  the real case.
 \end{remark}
 
By Remark \ref{rem::Bou_closed}, $diag(n, \QQ_p)$ is a closed subset of $\mathcal{M}(n, \QQ_p)$, and thus a closed subalgebra of $\mathcal{M}(n, \QQ_p)$. 
Similarly,   $\mathfrak{g}$ is a closed subset of $\mathcal{M}(n, \QQ_p)$. 

\begin{definition}
\label{def::closed_sl}
Let $V$ be a vector space over $\QQ_p$ of finite dimension $N$. Endow $V$ with the product topology from $\QQ_p ^{N}$.  We denote by $\mathcal{S} (V)$ the set of all $\QQ_p$-linear subspace of $V$ (thus closed). In particular, one can consider $V= \mathfrak{g}$.
\end{definition} 
 
We have $\mathfrak{c}$ is a closed and abelian subalgebra of $\mathfrak{g}$, so $\mathfrak{c} \in \mathcal{S}(\mathfrak{g})$.  For every $g \in G$, we obtain $g \mathfrak{c}g^{-1}$ is also a closed abelian subalgebra of $(\mathfrak{g}, [\cdot,\cdot])$, thus $Cart(\mathfrak{g}) \subset \mathcal{S} (\mathfrak{g})$.

Note $a, b \in \mathfrak{g}$ then $ab$ is not necessarily in $\mathfrak{g}$, and $[a,b]=0$ if and only if $ab=ba$.

\begin{lemma} 
\label{chab_g}
The set $\mathcal{S}(\mathfrak{g})$ is compact with respect to the Chabauty topology from $\mathfrak{g}$.
\end{lemma}

\begin{proof} First, $\mathfrak{g}$ is a locally compact  topological group (with matrix addition) with the topology induced from the locally compact topological group $\mathcal{M}(n, \QQ_p)$.  By Paulin~\cite[Proposition~1.7, p.~58]{CoPau} the space $\mathcal{F}(\mathfrak{g})$ of all closed subsets of $\mathfrak{g}$ is compact with respect to the Chabauty topology on $\mathcal{F}(\mathfrak{g})$. By~\cite[Proposition 17(2)]{CoPau}  $\mathcal{S}(\mathfrak{g})$ is a closed subset of $\mathcal{F}(\mathfrak{g})$ and thus compact with respect to the Chabauty topology.
\end{proof} 

Abelian groups and algebras satisfy the universal relation that the commutator is trivial, so limits of abelian groups and algebras are abelian. As a consequence of Proposition \ref{prop::Cooper} we obtain:

\begin{lemma}
\label{limits_abel} \begin{enumerate}
\item $\overline{Cart(\mathfrak{g}) }^{Ch}$  contains only abelian subalgebras of $(\mathfrak{g}, [\cdot,\cdot])$.
\item $\overline{Cart(G)}^{Ch}$ contains only abelian subgroups.
\end{enumerate} 
\end{lemma}

\subsubsection{\textbf{Defining the bijection $Gr: \overline{Cart(\mathfrak{g}) }^{Ch} \to \overline{Cart(G)}^{Ch}$}}
\begin{remark}
By Lemma \ref{stab} we may define the bijective map
 $$ Gr: Cart(\mathfrak{g}) \to  Cart(G) \qquad \qquad g \mathfrak{c} g ^{-1} \mapsto g C g^{-1} .$$
\end{remark}


Our goal is to establish a bijection between the closures $\overline{Cart(\mathfrak{g}) }^{Ch}$ and $\overline{Cart(G)}^{Ch}$.

\begin{definition}  Let $A \subset \mathcal{M}(n, \QQ_p)$ be a subalgebra of $\mathcal{M}(n, \QQ_p)$.  Denote by $A^*$ the group of all invertible elements of $A$ (i.e., $a \in A^*$ if and only if $a \in A \cap \GL(n, \QQ_p)$ with $a^{-1} \in A$). 
\end{definition}

\begin{lemma}
\label{lem::invers_elements}
Let $K$ be a field, in particular one can take $K = \QQ_p$. Let $A \subset \mathcal{M}(n, K)$ be a subalgebra of $\mathcal{M}(n, K)$. Then $A^* = A \cap \GL(n, K)$.
\end{lemma}
\begin{proof}
The inclusion $A^* \subset A \cap \GL(n, K)$ is clear. For the converse, let $a \in A \cap \GL(n, K)$. Let $P(t)= t^n+ k_{n-1} t^{n-1}+\cdots + k_1 t +k_0$ be the characteristic polynomial of $a$, where $k_{n-1}, \cdots, k_1, k_0 \in K$ with $k_0= \det(a) \neq 0$. By Cayley--Hamilton we have $0= P(a)=a^n+ k_{n-1} a^{n-1}+\cdots + k_1 a +k_0$ and so $a^{-1} = -\frac{1}{k_0} (a^{n-1}+ k_{n-1} a^{n-2}+\cdots + k_1 \Id )$ belongs to $A$.
\end{proof}

\begin{lemma}
\label{inv_gp}  
Let $A \in \mathcal{S}(\mathfrak{g})$ such that $\langle A, \Id \rangle$ is a subalgebra of $\mathcal{M}(n, \QQ_p)$.  Then $\langle A, \Id \rangle ^* = \langle A, \Id \rangle \cap \GL(n, \QQ_p)$ is a closed subgroup of $\GL(n, \QQ_p)$. 
\end{lemma}

\begin{proof} 
As by Lemma \ref{lem::invers_elements} $\langle A, \Id \rangle ^* = \langle A, \Id \rangle \cap \GL(n, \QQ_p)$ the conclusion follows.
\end{proof}

\begin{remark} 
\label{rem::Gr_cartan}
Notice $\langle \mathfrak{c}, \Id \rangle  \cap G= \langle \mathfrak{c}, \Id \rangle^*  \cap G = C$.   
\end{remark}

\begin{lemma}
\label{lem::Chab_cont_span}
Let $V$ be a finite dimensional vector space over $\QQ_p$.  Let $W$ be a hyperplane in $V$ and let $v$ be a vector in $V \setminus W$. Then the map $e_v : \mathcal{S}(W) \to \mathcal{S}(V)$ given by $A \mapsto \langle A, v \rangle$ is continuous with respect to the Chabauty topology. In particular, $e_v$ is injective.
\end{lemma}
\begin{proof}
We write $W = \ker \alpha$ where $\alpha$ is some linear form on $V$ such that $\alpha(v)=1$. 

Let $\{A_n\}_{n \in \NN}$ be a sequence of linear subspaces of $W$ converging in the Chabauty topology to a linear subspace $A$ of $W$. For $x= a+ \lambda v \in e_v(A)$ write $a= \lim_{n \to \infty} a_n $, with $a_n \in A_n$, for every $n \in \NN$. Then, for every $n \in \NN$ we have $a_n + \lambda v \in e_v(A_n)$ and $ \lim_{n \to \infty}(a_n + \lambda v )= x$. 

Conversely, if $\{x_n \in e_v(A_n)\}_{n \in \NN}$ converges to $x = a + \lambda v \in V$, then by writing $x_n= a_n + \lambda_n v$ and applying $\alpha$ we obtain $ \lim_{n \to \infty} \lambda_n =\lambda$. Then $\lim_{n \to \infty} a_n= \lim_{n \to \infty} (x_n - \lambda_n v)= x- \lambda v=a$. Thus $a \in A$ and $x \in e_v(A)$. We have shown the Chabauty limit of $\{e_v(A_n)\}_{n \in \NN}$ is $e_v(A)$. 

As $A = W \cap e_v(A)$ for every $A \in \mathcal{S}(W)$ the injectivity follows for $e_v$.
\end{proof}

\begin{remark}
\label{lem::6.15}
As a consequence of Lemma \ref{lem::Chab_cont_span} we obtain: if $\{ A_m \}_{m \in \NN} \subset \overline{Cart(\mathfrak{g})}^{Ch}$ with $A_m \xrightarrow{m \to \infty} A \in \overline{Cart(\mathfrak{g})}^{Ch}$,  then $\langle A_m , \Id \rangle \xrightarrow{m \to \infty}  \langle A, \Id \rangle$ with respect to the Chabauty topology induced from $\mathcal{M}(n, \QQ_p)$. 
\end{remark}

\begin{lemma}
\label{alg_limits_abel}  
Let $A \in Cart(\mathfrak{g})$, then $\langle A, \Id \rangle$ is an abelian subalgebra of $\mathcal{M}(n, \QQ_p)$.  If $\{A_m \}_{m \in \NN} \subset Cart (\mathfrak{g})$ with $A_m \xrightarrow{m \to \infty} A \in \overline{Cart (\mathfrak{g})}^{Ch}$ then $\langle A, \Id \rangle$ is an abelian subalgebra of $\mathcal{M} (n, \QQ_p)$.
\end{lemma} 

\begin{proof}  

Let $A \in Cart(\mathfrak{g})$. Then there exists $g \in G$ such that $A=g\mathfrak{c} g^{-1}$ and by Lemma~\ref{cid} $\langle A, \Id \rangle= g \cdot  diag(n, \QQ_p) \cdot g^{-1}$. 
So $\langle A, \Id \rangle$ is a subalgebra of $\mathcal{M}(n, \QQ_p)$ with respect to the multiplication of matrices. As $diag(n, \QQ_p)$ is an abelian subalgebra  this implies $\langle A, \Id \rangle$ is abelian too. 

Let $\{A_m \}_{m \in \NN} \subset Cart (\mathfrak{g})$ with $A_m \xrightarrow{m \to \infty} A \in \overline{Cart (\mathfrak{g})}^{Ch}$. Let $a, b \in \langle A, \Id \rangle$. By Remark \ref{lem::6.15} there exist $\{a_m \in \langle A_m, \Id \rangle\}_{m \in \NN}$ and $\{b_m \in \langle A_m, \Id \rangle\}_{m \in \NN}$ with $a_m  \xrightarrow{m \to \infty} a$ and $b_m  \xrightarrow{m \to \infty} b$. 

By the first part of the lemma we know that $a_m b_m \in \langle A_m, \Id \rangle$. So $a_m b_m= c_m + \lambda_m \Id$, where $c_m \in A_m$, $trace(c_m)=0$ and $trace(a_m b_m)=n \lambda_m$.

Moreover, we have $a_m b_m  \xrightarrow{m \to \infty} ab$, and thus $trace(a_m b_m)=n \lambda_m \xrightarrow{m \to \infty} trace(ab)=:n \lambda$. We write $ab= c+\lambda \Id $, with $c \in \mathfrak{g}$. Then $c_m \xrightarrow{m \to \infty}  c$ and as $A_m \xrightarrow{m \to \infty} A$ we obtain $c \in A$. It follows $ab \in \langle A, \Id \rangle$. Because $a_m b_m= b_m a_m$ we have $ab=ba$ and the conclusion of the lemma follows.
\end{proof} 

Note in the proof of Lemma \ref{alg_limits_abel}  it was important that $A \in Cart(\mathfrak{g})$ is a conjugate of the set of all diagonal matrices with trace zero. 

\medskip
 Let $\mathcal{S}_{1}(\mathcal{M}(n, \QQ_p))$ be the set of unital subalgebras of $\mathcal{M}(n, \QQ_p)$. 

\begin{proposition}
\label{prop::Gr_tilde}
The map $\tilde{Gr}: \mathcal{S}_{1}(\mathcal{M}(n, \QQ_p)) \to \mathcal{S}(G)$ given by $B \mapsto \tilde{Gr}(B):= B \cap G = B^* \cap G$ is Chabauty continuous.
\end{proposition}

In order to prove this we need the following lemma, which allows us to construct a continuous branching of the $m$-th root on a neighborhood of 1 in $\QQ_p^*$.

\begin{lemma}[pointed out by Y. de Cornulier] 
\label{lem::Y_C}
Fix $m \geq 1$. Then there exists an open neighbourhood $U$ of $1$ in $\QQ_{p}^{*}$ and a continuous function $g : U \to \QQ_{p}^{*} $ such that $g(x)^m = x$ for every $x \in U$.
\end{lemma}

\begin{proof}
By Robert \cite[Chapter 5, Section 4.1]{Rob00}, the function mapping $x \mapsto \exp(\frac{\log(1+x)}{m})$ is defined on a small neighbourhood $V$ around $0$ in $\QQ_p$; denote this function by $f$.  Moreover by \cite[Chapter 5, Section 4.2, Prop. 3]{Rob00} we have $f(x)^m = \exp \log(1+x) = 1+x$. Thus the function $g(x) = f(x-1) $ defined on $U=1+V$ does the job.
\end{proof}

\begin{proof}[Proof of Proposition \ref{prop::Gr_tilde}]
Let $\{B_k\}_{k \in \NN}$ be a sequence in $\mathcal{S}_{1}(\mathcal{M}(n, \QQ_p))$, converging to $B \in \mathcal{S}_{1}(\mathcal{M}(n, \QQ_p))$. We need to show the sequence $\{B_k \cap G\}_{k \in \NN}$ converges to $B \cap G$. Indeed, if a sequence $\{b_k \in B_k \cap G\}_{k \in \NN}$ converges to $b \in \mathcal{M}(n, \QQ_p)$, then clearly $b \in B \cap G$ as $G$ is closed. Conversely, if $b \in B\cap G$, then we may write $b=\lim_{k \to \infty} b_k$  with $b_k \in B_k$. So $\lim_{k \to \infty} \det(b_k) = 1$, and $\det(b_k)$  is in a neighbourhood $U \subset \QQ_p$ of $1 \in \QQ_p$, for $ k>0$ large enough. By Lemma \ref{lem::Y_C} (for $m=n$), define $ \lambda_k: = g(\det(b_k))$ (so that $\lim_{k \to \infty} \lambda_k =1$) and $b'_k = \lambda^{-1}_{k} b_k$. Then $b'_k \in B_k \cap G$ and $\lim_{k \to \infty}b'_k=b$.
\end{proof}

\begin{proposition}
\label{welldef} 
Let $A \in \overline{Cart (\mathfrak{g})}^{Ch}$, then $Gr(A) := \langle A, \Id \rangle \cap G$ is a closed subgroup of  $\overline{Cart (G)}^{Ch}$ and if $X\in A$ belongs to the domain of the exponential map $Exp$, then $Exp(X)\in Gr(A)$. The map $Gr:  \overline{Cart (\mathfrak{g})}^{Ch} \to \overline{Cart (G)}^{Ch}$ is well defined and  continuous with respect to the Chabauty topology on $\overline{Cart (\mathfrak{g})}^{Ch}$ and $\overline{Cart (G)}^{Ch}$.
\end{proposition} 

\begin{proof} 
By Lemma~\ref{alg_limits_abel}, $\langle A, \Id \rangle$ is a closed abelian subalgebra of $\mathcal{M}(n, \QQ_p)$. Thus by Lemma~\ref{inv_gp}, $Gr(A) = \langle A, \Id \rangle ^* \cap G= \langle A, \Id \rangle \cap G$ is a closed abelian subgroup of $\GL(n , \QQ_p)$.  If $X\in A$, then $Exp(X)\in G$ as $X$ is of trace 0, and on the other hand $Exp(X)\in \langle A, \Id \rangle$ as the latter is a closed sub algebra of $\mathcal{M}(n, \QQ_p)$. So $Exp(X)\in Gr(A)$.

Now consider the map $e_{\Id} : \mathcal{S}(\mathfrak{g}) \to \mathfrak{\mathcal{M}(n, \QQ_p)}$ given by $A \mapsto e_{\Id}(A)=\langle A, \Id \rangle$. By Lemma \ref{lem::Chab_cont_span} it is Chabauty continuous. As well is the canonical inclusion $\iota : \overline{Cart (\mathfrak{g})}^{Ch} \to \mathcal{S}(\mathfrak{g})$. It is then easy to see the image of $e_{\Id} \circ \iota$ is contained in $\mathcal{S}_{1}(\mathcal{M}(n, \QQ_p))$ (as $\langle A, \Id \rangle$ is a unital subalgebra). So we may compose with $\tilde{Gr}$ and obtain $Gr:= \tilde{Gr} \circ e_{\Id} \circ \iota$ that is well-defined and Chabauty continuous. Since $Gr$ maps $Cart (\mathfrak{g})$ to $Cart (G)$, just by continuity it maps $\overline{Cart (\mathfrak{g})}^{Ch}$ to $\overline{Cart (G)}^{Ch}$.
\end{proof}




Notice $Gr(\mathfrak{c})=C$. Our goal is to prove the map $Gr:  \overline{Cart (\mathfrak{g})}^{Ch} \to \overline{Cart (G)}^{Ch}$ is a bijection.

\begin{proposition}
\label{surj} 
The map $Gr:  \overline{Cart (\mathfrak{g})}^{Ch} \to \overline{Cart (G)}^{Ch}$ is surjective. 
\end{proposition} 

\begin{proof} Take $H \in \overline{Cart (G)}^{Ch}$.  Then there is a sequence $\{H_m \} \in Cart(G)$ with $H_m \xrightarrow{m \to \infty} H$.  Each $H_m = g_m C g_m ^{-1}$ for $g_m \in G$.  Set $\{A_m := g_m \mathfrak{c} g_m ^{-1} \in Cart(\mathfrak{g})\}_{m \in \NN}$ a sequence. Then $Gr(A_m)= H_m$.  Since $\overline{Cart (\mathfrak{g})}^{Ch}$ is compact, $A_m$ admits a convergent subsequence $A_{m_k} \xrightarrow{m \to \infty} A \in \overline{Cart (\mathfrak{g})}^{Ch}$.  As $H_{m_k}=Gr(A_{m_k}) \xrightarrow{m \to \infty} Gr(A)$, by Proposition~\ref{welldef}, the conclusion follows. 
\end{proof}

\begin{proposition} 
\label{inj} 
Let $B \in \mathcal{S}_{1}(\mathcal{M}(n, \QQ_p))$. Then we have $B =\langle \QQ_{p}^{*} (B \cap G) \rangle $. In particular, the maps $\tilde{Gr}: \mathcal{S}_{1}(\mathcal{M}(n, \QQ_p)) \to \mathcal{S}(G)$ and $Gr:  \overline{Cart (\mathfrak{g})}^{Ch} \to \overline{Cart (G)}^{Ch}$ are injective. 
\end{proposition} 

\begin{proof} 
 Because $Gr= \tilde{Gr} \circ e_{\Id} \circ \iota$ and  $e_{\Id}, \iota$ are injective, it is enough to prove $\tilde{Gr}$ is injective. This follows from the first part of the proposition that we now prove in four steps:
 
1. Recall by Lemma \ref{lem::invers_elements} we have $B^* = B \cap \GL(n , \QQ_p)$. 

2. We claim  $B^*$ is open in $B$. Indeed, it is enough to show that $\Id$ has an open neighbourhood in $B$ such that any of its elements is also contained in $B^*$. It is a general fact for Banach unital algebras that if $\vert \vert a - \Id \vert \vert < 1$ and by setting  $b := \Id - a$, then $a^{-1} = (\Id - b)^{-1} = \sum_{k=0}^{\infty} b^k$ is a norm-convergent series. The claim follows. 

Let $d$ be the dimension of the $\QQ_p$-linear subspace $B$. Then $B^*$ is a $d$-dimensional $p$-adic analytic group, hence $B\cap G = B^{*} \cap G = \ker(\det \vert_{B^*})$ is a $(d-1)$-dimensional $p$-adic analytic subgroup (see Serre \cite{Ser92}, Part II, Chapter 4, Section 5, Corollary, LG4.12).

3. We claim $\QQ_{p}^{*}(B\cap G)$ is open in $B^*$. Indeed, consider the analytic group homomorphism $\phi : (B^* \cap G) \times \QQ_{p}^{*} \to B^{*}$ given by 
$(g,\lambda) \mapsto \phi(g,\lambda) := \lambda g$. Its differential at the identity $\Id$ is
$$D\phi_{\Id} :(B\cap \mathfrak{g}) \times \QQ_{p} \to B, \  \ ï¿œ\ (X,\lambda ) \mapsto X + \lambda \Id.$$
Notice $D\phi_{\Id}$ has rank $d$ so, by the inverse function theorem (Serre \cite{Ser92}, Part II, Chapter 2, LG2.10), $\phi$ is a local analytic isomorphism, hence its image contains an open set.

4. Finally, in a finite dimensional vector space, the linear span of an open set is the whole space. Hence $B =\langle \QQ_{p}^{*} (B \cap G) \rangle $.
\end{proof} 

We have obtained a continuous bijective map from the compact set $\overline{Cart (\mathfrak{g})}^{Ch}$ to the compact set $\overline{Cart (G)}^{Ch}$, which concludes the proof of Theorem \ref{grbij}.  As the Chabauty topology on  $\overline{Cart (G)}^{Ch}$ is Hausdorff and $\overline{Cart (\mathfrak{g})}^{Ch}$ is compact we directly conclude the map $Gr$ is a homeomorphism.

\begin{corollary}
\label{cor::homeo}
The map $Gr:  \overline{Cart (\mathfrak{g})}^{Ch} \to \overline{Cart (G)}^{Ch}$ is a homeomorphism. 
\end{corollary}

\begin{remark}
\label{lem::roots_unity}
Let $\mu_n$ be the group of $n$-th roots of unity in $\QQ_p^*$. Let $H \in \overline{Cart (G)}^{Ch}$ and let $\lambda \in \mu_n$. If $g \in H$ then we claim $\lambda g \in H$. Indeed, by Corollary \ref{cor::homeo}, take $A_H \in \overline{Cart (\mathfrak{g})}^{Ch}$ such that $G \cap \langle A_H, \Id \rangle= Gr(A_H)=H$. Take $g \in H$ and let the unique $a \in A_H$ such that $a+\frac{trace(g)}{n}\Id= g$. Then $\lambda g = \lambda (a+\frac{trace(g)}{n}\Id)=  \lambda a+\frac{trace(\lambda g)}{n}\Id \in  \langle A_H, \Id \rangle.$ But $\det(\lambda g)=1$ so the conclusion follows.
\end{remark}

\medskip
\subsubsection{\textbf{Remarks on the dimension of a limit of the diagonal Cartan}}
So far we have shown that for every $H \in \overline{Cart (G)}^{Ch}$ there exists a unique $A_H \in \overline{Cart (\mathfrak{g})}^{Ch}$ such that $Gr(A_H) = H$. We end this subsection by showing that the dimension of every $A_H \in \overline{Cart (\mathfrak{g})}^{Ch}$ is $n-1$.

\medskip
We thank Thomas Haettel for the idea of the proof of the following Proposition.  Let $V$ be a $\QQ_p$-vector space of dimension $N$. Denote by $\mathcal{S}(V, m)$ the set of all $\QQ_p$-linear subspaces of $V$ of dimension $m <N$. Then $\mathcal{S}(V, m) \subset \mathcal{S}(V)$. In particular, one can consider $V = \mathfrak{g}$ and $m =n-1$.

\begin{proposition} \label{chab_gras}
We have $\mathcal{S}(V,m)$ is closed in $\mathcal{S}(V)$ and  $\mathcal{S}(V, m)$ equals the Grassmannians $Grass(m,V)$ as a set. Moreover, $\mathcal{S}(V, m)$ with the Chabauty topology is homeomorphic to the  Grassmannians $Grass(m, V)$  endowed with a quotient topology induced from the topological group $\GL(V, \QQ_p)$.
\end{proposition} 

\begin{proof} 
Recall $\mathcal{S}(V)$ is the set of all $\QQ_p$-linear subspaces of $V$ and by Lemma~\ref{chab_g} $\mathcal{S} (V)$ is compact.  Let $\GL(V, \QQ_p)$ be the set of invertible linear maps on $V$ considered as a $\QQ_p$-linear vector space.   Note $\GL(V,\QQ_p)$ acts by left multiplication and continuously on $\mathcal{S}(V)$ with respect to the Chabauty topology. Indeed,  suppose $\{D_k\}_{k \in \NN} \subset \GL(V, \QQ_p)$ converges to some $D \in \GL(V,\QQ_p)$ and $\{F_k\}_{k \in \NN} \subset \mathcal{S}(V)$ converges to $F \in \mathcal{S}(V)$. As $\GL(V,\QQ_p)$ acts continuously on  $V$ we have $\{D_k f_k\}_{k \in \NN}$ converges to $D f$, for every sequence $\{f_k \in F_k\}_{k \in \NN}$ that converges to $f \in F$. It remains to prove if $\{D_k f_k\}_{k \in \NN}$ admits a strictly increasing subsequence $\{k_l\}_{l \in \NN}$ with $\{D_{k_l} f_{k_l}\}_{l \in \NN}$ converging to some $g \in V$, then $D^{-1}g \in F$. As $D^{-1} D_{k_l}$ converges to $\Id \in \GL(V,\QQ_p)$, for $l$ large enough $f_{k_l}$ is as close as we want to $D^{-1} D_{k_l} f_{k_l}$, and thus to $D^{-1}g$. We obtain $\{f_{k_l}\}_{l \in \NN}$ converges to $D^{-1}g$ and so $D^{-1}g$ is in $F$. Therefore $\{D_k F_k\}_{k \in \NN}$ converges to $D F$ in the Chabauty topology on $\mathcal{S}(V)$.

Since we can apply a linear transformation in $\GL(V, \QQ_p)$ to change between bases of any two $\QQ_p$-linear subspaces of the same dimension, there are a finite number of  $\QQ_p$-linear subspaces of $V$, up to $\GL(V,\QQ_p)$ action. Thus the $\GL(V,\QQ_p)$-action on  $\mathcal{S} (V)$ is cocompact.  

Recall $\GL(V,\QQ_p)$ admits the Iwasawa decomposition  $KAN=KB$  (see Lemma \ref{lem::Iwasawa_decom}) where $K=\GL(V, \ZZ_p)$, $A$ is the set of diagonal matrices in $\GL(V,\QQ_p)$, $B$ is the subgroup of all upper-triagular matrices of $\GL(V, \QQ_p)$, and $N$ the set of all upper-triangular matrices with $1$ on the diagonal.  As in \cite[Section 1.1.3]{Lit} we fix the standard basis $\left\{e_1, \cdots, e_{N}\right\}$ of $V$, then the $\QQ_p$-linear subspace $F_{d}:=span\left\{e_1, \cdots, e_{d}\right\}$ of dimension $d$ in $V$ is invariant under the group that contains $B$:

$$\GL(V, F_{d}):=\begin{pmatrix} * & * \\
0_{(N-d)\times (N-d)} & * 
\end{pmatrix}.$$ 

As $F_{d}$ is in $\mathcal{S}(V, d)$, we obtain $\GL(V,\ZZ_p)$ acts transitively on $\mathcal{S}(V, d)$ which implies $\GL(V, \QQ_p)$ and $\GL(V,\ZZ_p)$ have the same orbits on $\mathcal{S}(V)$.  Since $\GL(V,\ZZ_p)$ is compact, every $\GL(V, \QQ_p)$-orbit in $\mathcal{S} (V)$ is closed.  Now, as $F_{m} \in \mathcal{S}(V, m) \subset  \mathcal{S} (V)$ and $F_{m} \in Grass(m, V)$, the orbit $\GL(V, \QQ_p) \cdot F_{m}$ is in bijection with $Grass(m, V)$ as a set.  Because elements of $\GL(V, \QQ_p)$ do not change the dimension of a $\QQ_p$-linear subspace, we obtain $\mathcal{S}(V, m) $ is closed in $\mathcal{S}(V)$ and in bijection with $Grass(m, V)$.

Denote the stabilizer of $F_{m}$ in $\GL(V, \QQ_p)$  by $\GL(V, F_{m})$. Then the map from $\GL(V, \QQ_p)/ \GL(V, F_{m})$ $\cong Grass(m, V)$ to $\mathcal{S}(V, m)$ given by $g \mapsto g F_{m}$ is bijective and continuous with respect to the quotient topology induced from the topological group $\GL(V, \QQ_p)$ and the Chabauty topology on $\mathcal{S}(V, m)$, that is also Hausdorff. We conclude $\mathcal{S}(V, m)$ with the Chabauty topology is homeomorphic to $\GL(V, \QQ_p)/ \GL(V, F_{m}) \cong Grass(m, V)$ with the quotient topology induced from the topological group $\GL(V, \QQ_p)$. 
\end{proof} 

\begin{corollary}
\label{cor::dim_same}
Every element of  $\overline{Cart (\mathfrak{g})}^{Ch}$ is an abelian subalgebra with respect to the Lie bracket of $\mathfrak{g}$ and of dimension $n-1$.
\end{corollary}


\subsection{Elliptic Cartan limits}

\begin{lemma}
\label{Borel}  
Take $H \in \overline{Cart (G)}^{Ch}$ and the corresponding $A_H \in \overline{Cart (\mathfrak{g})}^{Ch}$. If $H$ is contained in the upper triangular Borel subgroup $B \subset G$, then every matrix in $A_H$ is upper triangular. 
\end{lemma}

\begin{proof}  
We prove the contrapositive. Assume that $A_H$ contains some non-upper triangular matrix $a$. Consider the sequence $\{p^k a+\Id\}_{k>0}$ in $\langle A_H,\Id\rangle$.  As $k\rightarrow\infty$ the sequence $\{p^k a+\Id\}_{k>0}$ converges to $\Id$, hence $\det(p^ka+\Id)$ converges to 1. Fix $k$ large enough, so that $\det(p^k a+\Id)$ belongs to the neighborhood $U$ of 1 appearing in Lemma \ref{lem::Y_C}. Then the matrix $\frac{1}{g(\det(p^k a +  \Id))}( p^k a +  \Id )$ is in $G \cap \langle A_H, \Id \rangle = H$, and it is clearly not upper triangular.
%
%
\end{proof} 


\begin{remark}[See Platonov--Rapinchuk~\cite{PR94} page~151]
\label{rem::hyp_element}
For $\SL(n, \QQ_p)$ the ``good'' maximal compact subgroup is $\SL(n, \ZZ_{p})$, and so elements of $\SL(n, \ZZ_{p})$ are elliptic. The diagonal matrices in $\SL(n, \QQ_p)$ of the form $diag(p^{a_1}, \cdots, p^{a_n})$, with the condition that $(a_1,\cdots, a_n) \in \ZZ^{n}$ and $\sum\limits_{i=1}^{n} a_i=0$, are hyperbolic. The  Iwasawa and Levi decompositions associated with $\SL(n, \QQ_p)$ implies an upper triangular matrix of $\SL(n, \QQ_p)$ is hyperbolic if and only if at least one entry of its diagonal has $p$-adic absolute value $>1$.

\end{remark}

\begin{lemma}
\label{lem::p-adic_ineq} 
Let $\alpha_1, ..., \alpha_n \in \QQ_p$. Suppose for some fixed $i\neq j \in \{1,...,n\}$ we have $\alpha_i \neq \alpha_j$. Then there  exists a non-empty open subset  $U$ in $(\QQ_p, +)$ such that $\vert \alpha_i +\lambda\vert_p^n > \vert \alpha_j+\lambda\vert_p^n$ and $|\alpha_i+\lambda|_p^n > \prod_{t=1}^n |\alpha_t+\lambda|_p > 0$ for every $\lambda\in U$. 
\end{lemma}

\begin{proof} We prove the set $U$ is non-empty. This is done by considering two cases.

\textbf{Case $1$.} Suppose that $\alpha_j=0$. 

First notice that for every $\alpha_k$ that is not zero and for $m>0$ large enough  we have $|\alpha_k+p^m|_p= |\alpha_k|_p$. Moreover, for $m>0$ large enough we also have $\prod_{t:\alpha_t\neq 0}|\alpha_t+ p^m|_p= \prod_{t:\alpha_t\neq 0}|\alpha_t|_p$.

Thus, set $ \lambda = p^m \in \QQ_p$ for some $m>0$ large enough so that $p^{-m} < \vert \alpha_i\vert^n_p = |\alpha_i+\lambda|_p^n $ and $ \vert \alpha_i\vert^n_p > \prod_{t=1}^n |\alpha_t+\lambda|_p > 0$. Then $\vert \alpha_j+ \lambda \vert_p= p^{-m} $ and so $\vert \alpha_j + \lambda \vert_p^n <  \vert \alpha_i +\lambda\vert_p^n$. 

\textbf{Case $2$.} We now treat the general case when $\alpha_j \neq 0$.
Then we may apply the first case to the sequence $\alpha_1-\alpha_j,\alpha_2-\alpha_j,...,\alpha_n-\alpha_j$. The result follows. 

Finally, we observe the set, $U$, of $\lambda$'s in $(\QQ_p,+)$ satisfying all the inequalities in the statement of the lemma, is clearly open.
\end{proof}

As a consequence of  Lemma \ref{lem::p-adic_ineq} we have:

\begin{theorem} \label{ellip_unip} [Elliptic $\Rightarrow$ Unipotent] Given $H \in \overline{Cart (G)}^{Ch}$ take $A_H \in \overline{Cart (\mathfrak{g})}^{Ch}$ so that $Gr(A_H) = H$. If $H$ is contained in the Borel subgroup $B$, and $H$ does not contain hyperbolic elements, then every matrix in $A_H$ has only zero on the diagonal.  Moreover, $H$ is contained in $U^\emptyset \cdot \mu_n$, where $U^\emptyset$ is the unipotent radical of $B$ (see Example \ref{ex:levi_decom}) and $\mu_n$ is the group of $n$-th roots of unity in $\QQ_p^*$.
\end{theorem} 

\begin{proof} 
By Lemma \ref{Borel}, every matrix in $A_H$ is upper triangular. Fix $a\in A_H$, with diagonal values $a_1,...,a_n$. Suppose for contradiction that not all $a_i$'s are 0. By lemma \ref{lem::p-adic_ineq}, there exist $i\in\{1,...,n\}$ and $\lambda\in\QQ_p$ such that
$$|a_i+\lambda|_p^n > \prod_{j=1}^n |a_j+\lambda|_p > 0.$$
Then $\det(a+\lambda \Id)\neq 0$, hence $\frac{(a+\lambda \Id)^n}{\det(a+\lambda \Id)}$ belongs to $G$ and is a hyperbolic matrix, since the $i$-th diagonal coefficient has the $p$-adic norm $>1$. This contradicts the assumption on $H$, so all diagonal coefficients of $a$ are 0.

To prove the last statement of the theorem, write $H=Gr(A_H)$, with $A_H$ strictly upper triangular by the first statement. Then $\langle A,\Id\rangle$ consists of strictly upper triangular plus scalar matrices, so intersecting with $G$ one obtains a subgroup in $U^\emptyset \cdot \mu_n$. This proves the Theorem. 
\end{proof}


\subsection{Hyperbolic Cartan limits}
\label{subsec::hyper_Cartan}

In this section we prove that hyperbolic Cartan limits are, up to conjugacy, subgroups of $B \cap G_{\sigma_+} \cap G_{\sigma_-}$, for some subsimplex $\sigma_+$ of the ideal chamber $c$ and $\sigma_-$ is opposite $\sigma_+$ and contained in $\partial \Sigma$. 

\begin{lemma}
\label{lem::normal}
Let $H$ be a subgroup of $B$ and let $H^{0}:=\{h \in H \; \vert \; h \text{ is elliptic}\}$. Then $H^{0}$ is a normal subgroup of $H$. If $H$ is closed then $H^{0}$ is closed too. 
\end{lemma}

\begin{proof}
This follows from Lemma~3.1 from~\cite{CaCi}.
\end{proof}

\begin{lemma}
\label{lem::hype_borel}
Let $h \in B$ be elliptic and let $g \in B$ be hyperbolic. Then $gh$ is hyperbolic.
\end{lemma}
\begin{proof}
Suppose $gh$ is elliptic then $gh$ fixes a point in the affine building $X$. Thus $gh$ fixes pointwise an entire sector $Q$ with ideal boundary the ideal chamber $c$ (that defines the Borel subgroup $B$) and that is contained in the apartment $\Sigma$. As $h$ is also elliptic it fixes a sector $Q'$ in $\Sigma$ with ideal boundary the ideal chamber $c$.  Then $h$ and $gh$ fix pointwise $Q \cap Q'$.  So $g$ must fix pointwise $Q \cap Q'$.  This is a contradiction as $g$ is hyperbolic. It follows that $gh$ is hyperbolic.
\end{proof}

Recall by Lemma~\ref{lem::Iwasawa_decom} $B=B^0 A$ where $A \cong \ZZ^{n-1}$ is the set of all diagonal matrices in $G$ that are hyperbolic elements, but the $\Id$ (see Remark~\ref{rem::hyp_element}) and $B^0$ is normal in $B$. Then by Lemma~\ref{lem::hype_borel} for every $H \leq B$ the quotient $H/H^0$ contains only hyperbolic representatives. 

\begin{lemma}
\label{lem::H_hyp_abelian}
Let $H \leq B$ be a hyperbolic Cartan limit. Then $H/H^{0}$ is a free abelian subgroup of rank $m$, with $0 < m \leq n-1$, and maps to $A$ via a canonical injective homomorphism.
\end{lemma}
\begin{proof}
Since $H$ is a limit of the diagonal Cartan, it is abelian, and so is $H/ H^0$. As $H \leq B$ every $h \in H$ has a decomposition as $h=a_h h_1$ for a unique $a_h \in A$ and some $h_1 \in B^0$.  The element $a_h$ is unique because $A \setminus \left\{ \Id \right\}$ contains only hyperbolic elements. More precisely, $h=a_h h_1=b_h h_2$ implies $b_h^{-1}a_h \in B^0 \cap A$ so it is elliptic and must be the identity. Then the map $\psi :H/H^0 \to A $ given by $h \mapsto a_h$ is well defined and injective. Because $H$ is abelian and $H^0, B^0$ are normal subgroups in $H$, respectively, $B$, the map $\psi$ is also a homomorphism. We obtain $H/H^{0}$ is a free abelian subgroup of rank $0 < m \leq n-1$.
\end{proof}

The \textbf{convex hull} in $X$ of a flat $E \subset X$ is the intersection of all half-apartments and apartments in $X$ that contain $E$.  Denote the convex hull of $E$ in $X$ by $Conv_X(E)$ and notice it is unique. The convex hull is again a flat in $X$ and in general might be of bigger dimension that the dimension of $E$.

\begin{theorem}
\label{thm::hyper_Cartan}
Let $H$ be a hyperbolic Cartan limit with $\textrm{rank}( H / H^0) = m$. Then there exist $0 <k \leq n-1$ maximal (which depends on $m$) and 
$\sigma_+, \sigma_- \subset \partial \Sigma$ opposite simplices with $\sigma_+ \subset c$ such that up to conjugacy in $G$, $H$ is a subgroup of $B \cap G_{\sigma_+} \cap G_{\sigma_-}$ and stabilizes a flat of dimension  $k$ in $X$ whose ideal boundary is the support of $\sigma_+, \sigma_-$.
\end{theorem}

\begin{proof}
By Proposition~\ref{prop::limit_cartan_in_borel},  we may take $H$ to be a hyperbolic Cartan limit contained in $B$. By Lemma~\ref{lem::H_hyp_abelian} choose $m$ hyperbolic representatives in $H$ that generate $H/H^0$. Then those representatives generate a free abelian subgroup $H'$ of $H$ of rank $m \leq n-1$ and $H' \cong H/H^0$.  Apply the Flat Torus Theorem~\cite[Chapter II, Thm. 7.1]{BH99} and obtain $\Min(H'):= \bigcap\limits_{h \in H'} \Min(h)=Y \times E^{m}$, where $E^m$ is a flat in $X$ of dimension $m$, and every element $h \in H'$ stabilizes $\Min(H')$, the splitting $Y \times E^{m}$ and acts as a translation on $E^m$. In particular, $H'$ stabilizes $E^m$. 

By~\cite[Chapter II, Thm. 6.8 5)]{BH99} every element of $H$ stabilizes $\Min(H')$, preserves the splitting $Y \times E^{m}$  and acts as a translation on $E^m$.  In particular, every element in $H^0$ is elliptic and acts as the identity on $E^m$. We obtain $H$ stabilizes $Conv_X(E^m)$: more precisely, every hyperbolic element of $H$ acts as a translation on $Conv_X(E^m)$ and every elliptic element fixes pointwise $Conv_X(E^m)$. Then we take $k$ to be the dimension of $Conv_X(E^m)$ and notice $k$ is maximal. In particular, $H$ fixes pointwise the boundary $\partial Conv_X(E^m)$.

Then there exist two opposite ideal simplices (not necessarily unique) $\sigma'_+, \sigma'_- \subset \partial (Conv_X(E^m)) \cap \partial X$ such that the support of $\sigma'_+$ is $\partial (Conv_X(E^m) )$. Fix such  $\sigma'_+, \sigma'_-$. Since $H$ fixes pointwise the ideal chamber $c$ corresponding to the Borel subgroup $B$ and also the ideal simplex $\sigma'_+$,  and since the convex hull is unique, then $H$ fixes pointwise the convex hull $Conv_{\partial X}(\sigma'_+,c)$ in $\partial X$ of $c$ and $\sigma'_+$ (see \cite[Example~3.133.c)]{AB}). Let $c'$ be the unique ideal chamber of $Conv_{\partial X}(\sigma'_+,c)$ that contains the simplex $\sigma'_+$ (see \cite[Def.~3.104 and Exe. 3.149)]{AB}). Then by the strong transitivity of $G$ on $\Ch(\partial X)$ there exists $g \in G$ with $g(c')=c$ and $g(\sigma'_-) \in \partial \Sigma$. Thus there exist $\sigma_+ =g(\sigma'_+)  , \sigma_- = g(\sigma'_-)  \subset \partial \Sigma$ opposite simplices and $\sigma_+ \subset c$ such that, up to conjugacy, $H$ is a subgroup of $B \cap G_{\sigma_+} \cap G_{\sigma_-}$.  
\end{proof}

Theorems \ref{thm::hyper_Cartan} and \ref{ellip_unip} imply Theorem \ref{hyp_ellip}. 

\medskip
In order to give explicit examples of limits of the diagonal Cartan in low dimensions we need to understand further the diagonal entries of  matrix elements of hyperbolic Cartan limits. This is a generalization of Theorem \ref{ellip_unip}. 



\begin{theorem} 
\label{rem::hyp_Cartan}
Up to conjugacy, every hyperbolic Cartan limit $H$ is a subgroup of  $$B \cap  G_{\sigma_+} \cap G_{\sigma_-}= \left\{\begin{pmatrix} A_1 & 0 & \cdots & 0 \\
0 & A_2 & 0 & \cdots \\
0 & 0 & \ddots & 0\\
0& \cdots & 0 & A_k\\
\end{pmatrix} \in \SL(n, \QQ_p) \right\},$$ for some opposite simplices $\sigma_+, \sigma_- \subset \partial \Sigma$ with $\sigma_+ \subset c$ and where the blocks $A_1, \cdots, A_k$ are indecomposable upper triangular square matrices of possibly different dimensions, and in each block $A_i$ every element in $H$ has its diagonal entries all the same. 

In particular, if the dimension of $Conv_X(E^m)$  corresponding to $H$ is $n-1$, then $H$ is conjugate to the diagonal Cartan subgroup. 
\end{theorem} 

\begin{proof}
Recall a block $A_i$ is \textbf{indecomposable} if  no conjugate of $P^I$ in $SL (n,\QQ_p)$ allows us to write $A_i$ as a direct sum of smaller blocks

The first part of the proposition is just Theorem \ref{thm::hyper_Cartan},  Proposition \ref{coro::geom_levi_decom} and Example \ref{ex:levi_decom}. 

As $H \leq B$ and by Lemma \ref{Borel}, every matrix in $A_H$ is upper triangular, and so is every matrix in $H$.
Let us prove  in each block $A_i$ every matrix element in $H$ has its diagonal entries all the same. 

We treat hyperbolic and elliptic elements of $H$ separately. Let $F$ be the flat in $X$ of maximal dimension $k$ that is stabilized by $H$ and given by Theorem \ref{thm::hyper_Cartan}. Recall, from the proof of Theorem \ref{thm::hyper_Cartan}, every hyperbolic element $h$ of $H$ acts as a translation on $F$ and every elliptic element in $H$ fixes pointwise $F$. Then by Lemma \ref{rem::SL_Levi_factor} and Remark \ref{rem::M_I_elements} for every hyperbolic element $h \in H$ the preimage $\alpha^{-1}(\alpha\vert_{\Isom(F)}(h))$ consists of diagonal matrices with entries in each block $A_i$ all the same. Moreover, every elliptic element of $H$ has its diagonal entries in $\ZZ_p^*$, otherwise it will be hyperbolic (see Remark \ref{rem::shape_hyp_elements}). 

Suppose for contradiction that there is an elliptic element $h$ of $H$ such that in one block $A_i$ it has two different diagonal entries, say $\alpha_j \neq \alpha_l$; we fix for what follows $A_i, \alpha_j \neq \alpha_l$. Then $a:= h- \frac{\tau}{n} \Id \in A_H$, where $\tau = trace(h) \in \ZZ_p$, still has nonequal diagonal entries $\alpha_j- \frac{\tau}{n} \neq \alpha_l-\frac{\tau}{n}$. 

Denote the diagonal entries of $a$ by $a_1, ..., a_n$ where $a_j \neq a_l$ by the previous paragraph.  Then by Lemma \ref{lem::p-adic_ineq} there exists $\lambda \in \QQ_p$ such that $\vert a_j +\lambda\vert_p^n < \vert a_l+\lambda\vert_p^n$ and $|a_l+\lambda |_p^n > \prod_{t=1}^n | a_t+\lambda|_p > 0$, thus
 $$h_\lambda:= \frac{ (a + \lambda \Id)^n} { det ( a + \lambda \Id)} \in H$$ 
 is a hyperbolic element. By the first part of the proof, the preimage $\alpha^{-1}(\alpha\vert_{\Isom(F)}(h_{\lambda}))$ consists of diagonal matrices with entries in each block $A_t$ all the same. Take some $g_\lambda \in \alpha^{-1}(\alpha\vert_{\Isom(F)}(h_{\lambda}))$. As $|a_l+\lambda |_p^n > \prod_{t=1}^n | a_t+\lambda|_p > 0$, the corresponding entry $\lambda_i$ of $g_\lambda$ in the block $A_i$ is such that $\vert \lambda_i \vert_p>1$. Moreover,  we claim $g_\lambda^{-1} h_\lambda$ is an elliptic element of $B$. Indeed, applying the map $\alpha$ from Proposition \ref{prop::res_building_str_tran} to $g_\lambda^{-1} h_\lambda$ and by the choice of $g_\lambda$ we obtain $\alpha(g_\lambda^{-1} h_\lambda)=\alpha\vert_{\Isom(\Delta_I)}(h_\lambda)$. So by definition $g_\lambda^{-1} h_\lambda$ acts trivially on the flat $F$. As  $h_\lambda$ is in $B$ and $g_\lambda^{-1}$ is diagonal the claim follows. Thus the diagonal entries of $g_\lambda^{-1} h_\lambda$ must be in $\ZZ_p^{*}$. But the $j$-diagonal entry of $g_\lambda^{-1} h_\lambda$ in the block $A_i$, is such that  $\vert \frac{(a_j+\lambda)^n}{det ( a + \lambda \Id)} \cdot \frac{1}{\lambda_i} \vert_p <1$. This contradicts the fact that the diagonal entries of $g_\lambda^{-1} h_\lambda$ are in $\ZZ_p^*$.
 

If the dimension of $Conv_X(E^m)$ of the flat $E^m$ corresponding to $H$ is $n-1$, then $\sigma_+$ is the ideal chamber $c$. So $H$ is contained (up to conjugacy) in $C$. As the corresponding Lie algebra $A_H$ of $H$ is of dimension $n-1$ and thus contained in the set of diagonal matrices with trace zero, we conclude by Propositions~\ref{inj},~\ref{surj} that $H$ is a diagonal Cartan subgroup.
\end{proof}

\section{Explicit examples of limits of the diagonal Cartan in low dimensions}\label{lowdim}
Recall $\mu_k$ is the group of $k$th roots of unity in $\QQ_p$. 

\subsection{$\SL (2,\QQ_p)$}

We compute a limit of the diagonal Cartan.  By Lemma \ref{lem::unip_conj} it is sufficient to conjugate $C$ only by sequences in the unipotent radical $U$ of $B$. In view of Corollary \ref{cor::homeo}, we may replace $\overline{Cart (G)}^{Ch}$ by $\overline{Cart (\mathfrak{g})}^{Ch}$ and $C$ by $\mathfrak{c}$. Observe that

$$\bpmat 1 & p^{-n} \\0 & 1 \epmat 
\bpmat a_n & 0 \\ 0 & -a_n \epmat 
\bpmat 1 &p^ {-n} \\0 & 1 \epmat  ^ {-1} 
= 
\bpmat a_n & -2p^{-n}a_n\\ 0 & -a_n \epmat.$$
For this to converge to 
$
\bpmat 0 & x \\ 0 & 0 \epmat$, it is enough to take $a_n=\frac{-p^nx}{2}$. This computation shows that $\mathfrak{n}=\{\bpmat 0 & x \\ 0 & 0 \epmat: x\in\QQ_p\}$ is a Chabauty limit of conjugates of $\mathfrak{c}$. 
\newline
 Applying $Gr$, we obtain that  $\{\pm \Id\}.U$ is an elliptic Cartan limit.  We now show this is the only possible limit, up to conjugacy.

\begin{proposition}\label{sl2}  Up to conjugacy, there is only one limit of $C$ in $ \overline{Cart(\SL(2,\QQ_p))}^{Ch} \setminus Cart(\SL (2,\QQ_p))$: the product of the unipotent radical of the Borel subgroup $B$, with $\{\pm \Id\}$.
\end{proposition} 
\begin{proof} 
Recall  the Bruhat--Tits building of $\SL (2,\QQ_p)$ is a $(p+1)$-regular tree $T_{p+1}$ and up to conjugacy there is only one parabolic subgroup, the Borel subgroup $B$, equal to the set of the upper triangular matrices of $\SL(2,\QQ_p)$ (see Example~\ref{ex:levi_decom}). 

Let $H$ be a limit of the diagonal Cartan. By Proposition~\ref{prop::limit_cartan_in_borel} it is enough to consider $H  \leq \{\pm Id\}\cdot B$. 

By Theorem~\ref{ellip_unip}, if $H$ is an elliptic Cartan limit, $H$ is contained in the unipotent radical of $B$ times the square-roots of unity $\{\pm \Id\}$. Moreover, by Propositions~\ref{surj},~\ref{inj} and Corollary~\ref{cor::dim_same}  the  subalgebra $A_{H} \in \overline{Cart(\mathfrak{sl}(2,\QQ_p))}^{Ch} $ corresponding to $H$ is of dimension one and is a subset of all upper triangular matrices in $\mathfrak{sl}(2,\QQ_p)$ with only zero on the diagonal. In other words $A_H$ is the set of strictly upper diagonal nilpotent matrices. Then $H=\langle A_H, \Id \rangle^ * = \{\pm \Id\}.U$, with $U$ the unipotent radical of $B$.
 
 If $H$ is a hyperbolic Cartan limit by Theorem \ref{rem::hyp_Cartan}  $H$ is in $Cart(\SL(2,\QQ_p))$.
\end{proof} 


\medskip

We obtain geometric intuition using the lattice construction of the Bruhat--Tits building $X= T_{p+1}$ of $\SL(2, \QQ_p)$ from \cite[Chapter~19]{Gar97}.  The diagonal Cartan subgroup stabilizes an apartment, which in the tree $T_{p+1}$ is a bi-infinite line. A vertex in $T_{p+1}$ is an equivalence class of lattices,  and by choosing the standard basis for $\QQ_p^2$, the vertices of the a bi-infinite line stabilized by $C$ are determined by the chosen base in $\QQ_p^2$.    Going to infinity can be seen by moving the second basis vector onto the first basis vector, so that the angle between them goes to zero.  They collapse to a single generalized eigenvector which is preserved by the limit group under this same sequence.

\subsection{$\SL(3,\QQ_p)$}


By \cite[Chapter~19]{Gar97} the Bruhat--Tits building $X$ for $\SL(3,\QQ_p)$ is constructed by gluing triangles with $p+1$ triangles along each edge.  Vertices correspond to lattices in $\QQ _p^3$.  To go to infinity we move the generators of the lattice close together in different ways (which gives us non-conjugate limits of $C$). To classify limits in $\SL (3,\QQ_p)$ we use: 

 \begin{remark}\label{cubeclass} By \cite[Th\'eor\`eme 2, page 33]{serre} the structure of $\QQ_p^*$ is given by:
$$\QQ_p ^* \cong \begin{cases} \ZZ \times \ZZ_p  \times \ZZ/ (p-1) \ZZ & p \neq 2\\
\ZZ \times \ZZ_2  \times \ZZ / 2 \ZZ  & p =2.
\end{cases} $$
If $p \neq 3$ then $\ZZ_p$ is a $3$-divisible group, in which every element is a cube. From this it is easy to compute 
$$\cube = |\QQ_p ^*/ {\QQ_p ^*}^3|= \begin{cases}
 9 & p =3 \\
 9 & p \equiv 1 \pmod 3 \\
 3 & p \equiv 2 \pmod 3. 
\end{cases} $$
\end{remark}

\begin{proposition}\label{sl3} Up to conjugacy, there are $3 + \cube$  limits of $C$ in $ \overline{Cart(\SL  (3,\QQ_p))}^{Ch} \setminus Cart(\SL(3,\QQ_p))$: the product of $\mu_3$ with 
$$\begin{pmatrix} a & x & 0 \\
0 & a & 0 \\
0 & 0 & \frac{1}{a^2} 
\end{pmatrix}  \qquad
\begin{pmatrix} 1 & x & y \\ 
0 & 1 & \alpha x\\
 0 & 0 & 1
\end{pmatrix} \qquad
\begin{pmatrix} 1 & x & y \\ 
0 & 1 & 0\\
 0 & 0 & 1
\end{pmatrix} \qquad
\begin{pmatrix} 1 & 0 & y \\ 
0 & 1 & x\\
 0 & 0 & 1
\end{pmatrix}. $$
where $\alpha \in \QQ_p$ is fixed. 
\end{proposition} 

We will need the following:

\begin{lemma}\label{Nalpha} Set $N_\alpha=\{\begin{pmatrix} 1 & x & y \\ 
0 & 1 & \alpha x\\
 0 & 0 & 1
\end{pmatrix}:x,y\in\QQ_p\}$. In $\SL(3,\QQ_p)$, the subgroups $N_\alpha$ and $N_\beta$ are conjugate if and only $\alpha$ and $\beta$ are in the same class of $\QQ_p ^*/ {\QQ_p ^*}^3$.
\end{lemma}

\begin{proof} If $\alpha$ and $\beta$ are in the same class of $\QQ_p ^*/ {\QQ_p ^*}^3$, then 
$$\scriptsize{ \begin{pmatrix} (\frac{\beta}{\alpha})^{\frac{1}{3}} & 0 & 0 \\
0 & (\frac{\beta}{\alpha})^{\frac{1}{3}}& 0 \\
0 & 0 & (\frac{\beta}{\alpha})^{-\frac{2}{3}}
\end{pmatrix}} $$
conjugates $N_\alpha$ into $N_\beta$. Conversely, assume that $N_\alpha$ and $N_\beta$ are conjugate by some $T\in \SL(3,\QQ_p)$. Then $T$ also conjugates the Lie algebras $\mathfrak{n}_\alpha$ and $\mathfrak{n}_\beta$, where $\mathfrak{n}_\alpha=\{X_\alpha=\begin{pmatrix} 0 & x & y \\ 
0 & 0 & \alpha x\\
 0 & 0 & 0
\end{pmatrix}:x,y\in\QQ_p\}$. Let $\{e_1,e_2,e_3\}$ be the canonical basis of $\QQ_p^3$. Since $\langle e_1\rangle = \cap_{X_\alpha\in\mathfrak{n}_\alpha}\ker(X_\alpha)$, and $\mathfrak{n}_\alpha$ and $\mathfrak{n}_\beta$ are conjugate by $T$, we see that $\langle e_1\rangle$ is $T$-invariant. Similarly, since $\langle e_1,e_2\rangle = \cap_{X_\alpha\in\mathfrak{n}_\alpha}\ker(X_\alpha^2)$, we see that $\langle e_1,e_2\rangle$ is $T$-invariant. In other words, $T$ is upper triangular, say $T= \begin{pmatrix} \lambda & \gamma & \epsilon \\ 
0 & \mu & \delta\\
 0 & 0 & \frac{1}{\lambda\mu}
\end{pmatrix}$. Then by direct computation $TX_\alpha T^{-1}= \begin{pmatrix} 0 & \frac{\lambda x}{\mu} & \star \\ 
0 & 0 & \alpha \lambda\mu^2 x\\
 0 & 0 & 0
\end{pmatrix}$. For this to be in $\mathfrak{n}_\beta$, we need $\alpha\mu^2=\frac{\beta}{\mu}$, i.e., $\frac{\beta}{\alpha}$ is a cube in $\QQ_p^*$.
\end{proof}

\begin{proof}[Proof of \ref{sl3}]
Let $H$ be a limit of the diagonal Cartan. By Proposition~\ref{prop::limit_cartan_in_borel} it is enough to consider $H  \leq \mu_3\cdot B$, where $B$ is the minimal parabolic subgroup that equals the upper triangular matrices of $\SL (3,\QQ_p)$ (see Example~\ref{ex:levi_decom}).

By Section~\ref{subsec::hyper_Cartan}, if $H$ is a hyperbolic Cartan limit  $H$ either stabilizes a $2$-dimensional flat (that is an apartment) or a $1$-dimensional flat. The former case gives a diagonal Cartan subgroup. By Section~\ref{subsec::hyper_Cartan}  and Theorem~\ref{rem::hyp_Cartan}  in the latter case there are three possible subgroups 

$$\scriptsize{\begin{pmatrix} a & x &0\\
0 &a&0\\
0&0& a^{-2}
\end{pmatrix}
\qquad
\begin{pmatrix} a^{-2} & 0 &0\\
0 &a&x\\
0&0&a
\end{pmatrix}
\qquad
\begin{pmatrix} 
a & 0 & x\\
0 & a^{-2} & 0 \\
0 & 0 & a
\end{pmatrix} }
$$ where $a \in \QQ_p^*$ and $x \in \QQ_p$.  But these are all conjugate by permutation matrices.  To show that the first appears as a limit of the diagonal Cartan, we proceed as in Proposition \ref{sl2} and work at the level of $\overline{Cart(\mathfrak{sl}(3,\QQ_p))}^{Ch} $.

$$\scriptsize{ \begin{pmatrix} 1 & p^{-n } & 0\\
0 &1&0\\
0&0&1
\end{pmatrix} 
\begin{pmatrix} x_n & 0 & 0\\
0 &y_n &0\\
0&0&-(x_n+y_n)
\end{pmatrix}
\begin{pmatrix} 1 & p^{-n }&0\\
0 &1&0\\
0&0&1
\end{pmatrix} ^{-1} 
=
\begin{pmatrix} x_n & p^{-n}(y_n-x_n) &0\\
0 &y_n&0\\
0&0& -(x_n+y_n)
\end{pmatrix}.}
$$
For this to converge to 
$$\scriptsize{ \begin{pmatrix} a & x & 0\\
0 &a&0\\
0&0&-2a
\end{pmatrix} ,}$$
we may take $x_n=a$ and $y_n=a+p^nx$. This takes care of the unique (up to conjugacy) non-trivial hyperbolic Cartan limit.

By Theorem~\ref{ellip_unip}, if $H$ is an elliptic Cartan limit, then $H$ is contained in the product of $\mu_3$ with the unipotent radical of $B$. By Propositions~\ref{surj},~\ref{inj} and Corollary~\ref{cor::dim_same} the subalgebra $A_{H} \in \overline{Cart(\mathfrak{sl}(3,\QQ_p))}^{Ch} $ corresponding to $H$  is of dimension two and a subset of the upper triangular matrices in $\mathfrak{sl}(3,\QQ_p)$ having only zero on the diagonal. So $\langle A_H, \Id \rangle^ * = H$. It is then easy to see that the abelian subgroup $H$ is conjugate to the product of $\mu_3$ with one of the following three subgroups, with $x, y \in \QQ_p$,

$$\scriptsize N_{\alpha}= { \bpmat 1 & x & y \\
0 & 1 & \alpha x \\
0 & 0 & 1 \epmat
\qquad 
 \bpmat 1 & x & y \\
0 & 1 & 0 \\
0 & 0 & 1 \epmat
\qquad 
  \bpmat 1 & 0 & y \\
0 & 1 & x\\
0 & 0 & 1 \epmat}$$
and where $\alpha \in \QQ_p$ is fixed.  

 By lemma \ref{Nalpha} the number of conjugacy classes of $N_ \alpha$ is $| \QQ_p ^*/ {\QQ_p ^*}^3 | = \cube $. 


We claim that all the above three subgroups are limits of $C$. Indeed, consider the case of $N_\alpha$, and the following conjugating matrix:
$$ \scriptsize{\begin{pmatrix}  \alpha & p^{-n} & \frac{1}{2} p^{-2n}\\
0 &1& p^{-n}\\
0&0&  \frac{1}{\alpha}
\end{pmatrix} } $$
Again work at the level of $\overline{Cart(\mathfrak{sl}(3,\QQ_p))}^{Ch} $; then
$$\scriptsize{ \begin{pmatrix} \alpha & p^{-n } & \frac{p^{-2n}}{2}\\
0 &1&p^{-n}\\
0&0&\frac{1}{\alpha}
\end{pmatrix} 
\begin{pmatrix} x_n & 0 & 0\\
0 &y_n &0\\
0&0&-(x_n+y_n)
\end{pmatrix}
\begin{pmatrix} \alpha & p^{-n }&\frac{p^{-2n}}{2}\\
0 &1&p^{-n}\\
0&0&\frac{1}{\alpha}
\end{pmatrix} ^{-1} 
=
\begin{pmatrix} x_n & p^{-n}(y_n-x_n) & \frac{-3\alpha p^{-2n}y_n}{2}\\
0 &y_n&-\alpha p^{-n}(x_n+2y_n)\\
0&0& -(x_n+y_n)
\end{pmatrix}.}
$$
For this to converge to 
$$\scriptsize{ \begin{pmatrix} 0 & x & y\\
0 &0&\alpha x\\
0&0&0
\end{pmatrix} ,}$$
we may take $x_n=-p^n x$ and $y_n=\frac{-2p^{2n}y}{3\alpha}$. The reader can verify that the two remaining subgroups are obtained as limits of $C$ using the following conjugating sequences, respectively:

$$\scriptsize{\begin{pmatrix} 1 & p^{-n } & p^{-n}\\
0 &1&0\\
0&0&1
\end{pmatrix} ,
\begin{pmatrix} 1 & 0 & p^{-n}\\
0 &1&p^{-n}\\
0&0&1
\end{pmatrix}.} 
$$ 
\end{proof}

\begin{remark}  The sequences of conjugating matrices follow a geometric pattern.  The Cartan subgroup $C$ naturally preserves a projective triangle in $\mathbb{P} (\QQ_p ^3)$.  The limits of $C$ correspond to different ways of identifying points and lines in the triangle, given by projective transformations which are the same as the conjugating matrices.   For a full explanation, see \cite{Leitnersl3}. 
\end{remark}


\begin{example} 

 Recall $\SL(3,\mathbb{F}_2)$ is a finite simple Lie group, whose affine Weyl group is the dihedral group of order
$6$. Its associated Bruhat--Tits building is a finite spherical building that is represented by the Heawood graph. There are two kinds of vertices in the Heawood graph:  blue vertices corresponding to a point in $\PP( \mathbb{F}_2 ^3)$, and the red vertices corresponding  to a line in $\PP( \mathbb{F}_2 ^3)$. An apartment is a hexagon (i.e., $6$ connected vertices) in the Heawood graph and every half apartment  is represented as 4 connected vertices in the graph.

In the  Bruhat--Tits building of $\SL(3,\QQ_2)$ every vertex has a link (\cite[Def. A.19.]{AB}) which is represented by the Bruhat--Tits building for $\SL(3,\mathbb{F}_2)$ (\cite[Prop. A.20]{AB}), which is the Heawood graph. Moreover, a spherical apartment in the spherical building at infinity for $\SL(3,\QQ_2)$ is represented by an apartment in the Heawood graph. The red and blue types of vertices of the Heawood graph are stabilized by the two maximal parabolic subgroups of $\SL(3,\QQ_2)$. Intuitively, these red and blue vertices are  ``associated'' with the last two groups in the list of Proposition \ref{sl3}. The second group in the list of Proposition \ref{sl3} preserves the maximal flag $\langle e_1 \rangle \subset \langle e_1, e_2 \rangle \subset \langle e_1, e_2, e_3 \rangle $, which appears as the edge connecting a red and a blue vertex in the Heawood graph. Note, by Proposition~\ref{prop::limit_cartan_in_borel}  all these three groups are contained in the Borel subgroup, that stabilizes an edge=chamber in the spherical building at infinity for $\SL(3,\QQ_2)$. This ideal chamber corresponds to an edge in the Heawood graph. 

Finally, the first group preserves a wall in the Bruhat--Tits building of $\SL(3,\QQ_2)$ (that is a line). The two ideal opposite endpoints of this line are a red and a blue vertex; they can be visualized as two non-connected vertices in an apartment of the Heawood graph.

In the Bruhat--Tits building associated with $\SL(3,\QQ_p)$ where $p >2$ the links of vertices will be larger graphs. 
\end{example} 

Compare these results with similar results by \cite{Htt_2, Leitnersl3, ST} over $\RR$ and $\CC$.   The proofs written by \cite{Htt_2, ST} are not geometric, and rely heavily on computation using sequences.  The results here use the geometry of the affine and spherical buildings. 

\begin{figure}
\begin{tikzpicture}[style=thick, scale=0.83]
\draw (360/14:2cm) -- (2*360/14:2cm) -- (3*360/14:2cm) -- (4*360/14:2cm) -- (5*360/14:2cm) -- (6*360/14:2cm) -- (7*360/14:2cm) -- (8*360/14:2cm) -- (9*360/14:2cm) -- (10*360/14:2cm) -- (11*360/14:2cm) -- (12*360/14:2cm) -- (13*360/14:2cm) -- (14*360/14:2cm) -- cycle;
\draw (360/14:2cm) -- (6*360/14:2cm) ; 
\draw (2*360/14:2cm) -- (11*360/14:2cm) ; 
\draw (3*360/14:2cm) -- (8*360/14:2cm) ; 
\draw (4*360/14:2cm) -- (13*360/14:2cm) ; 
\draw (5*360/14:2cm) -- (10*360/14:2cm) ; 
\draw (9*360/14:2cm) -- (14*360/14:2cm) ; 
\draw (7*360/14:2cm) -- (12*360/14:2cm) ; 
\foreach \x in {1*360/14, 3*360/14, 5*360/14, 7*360/14, 9*360/14, 11*360/14, 13*360/14 }{\draw(\x: 2cm) circle(2pt) [red, fill=red] ; }
\foreach \x in {2*360/14, 4*360/14, 6*360/14, 8*360/14, 10*360/14, 12*360/14, 14*360/14 }{\draw(\x: 2cm) circle(2pt) [blue, fill=blue] ; }
\draw (360/14:2cm) node[right] {$\scriptsize{\langle e_1 \rangle, \bpmat 1 & s& t \\ 0 & 1 & 0 \\0 & 0 & 1 \epmat }$};
\draw (6*360/14:2cm) node[left] {$ \scriptsize{\bpmat 1 & 0& t \\ 0 & 1 & s \\0 & 0 & 1 \epmat , \langle e_1, e_2 \rangle }$};
\draw (8.5*360/14:3.5cm) node {$\scriptsize{\begin{array}{c} \bpmat 1 & s& t \\ 0 & 1 & s \\0 & 0 & 1 \epmat \\  \{ \langle e_1 \rangle , \langle e_1, e_2 \rangle \}  \end{array} }$};
\draw[->] (8.5*360/14:2.3cm) -- (8.5*360/14:2cm) ;
\end{tikzpicture}
\caption{Groups preserving faces in the Heawood graph}
\end{figure}
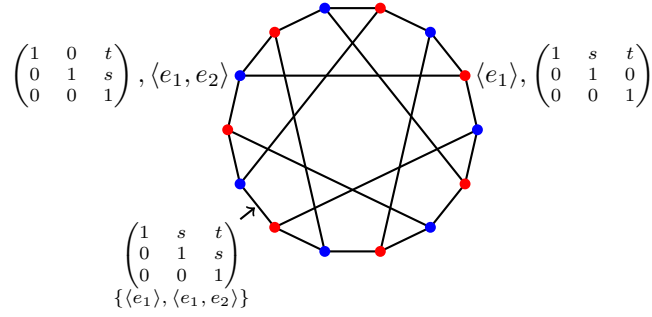 

\subsection{$\SL (4,\QQ_p)$}


For $k\geq 2$, set $\mathcal{Q}_k:= \vert \QQ_p ^* / {\QQ_p ^*}^k \vert$. Recall from \cite[section 3.3]{serre} that $\mathcal{Q}_2= 4$ for $p$ an odd prime, while $\mathcal{Q}_2=8$ for $p=2$.

 By Remark \ref{cubeclass} we see 
$$\mathcal{Q}_4= \left\{\begin{array}{ccc}8 & if & p\equiv 3 (\mod \;4) \\16 & if & p\equiv 1 (\mod \;4)  \\32 & if & p=2\end{array}\right.$$
$$\mathcal{Q}_8= \left\{\begin{array}{ccc}16 & if & p\equiv 3,7 (\mod \;8) \\32 & if & p\equiv 5 (\mod \;8) \\64 & if & p\equiv 1 (\mod \;8) \\128 & if & p=2\end{array}\right.$$

\begin{proposition} \label{sl4} 
The following families product with $\mu _4$ 
are the conjugacy classes of limits of the diagonal Cartan in $\SL(4,\QQ_p)$:
$$
\scriptsize{\begin{array}{cccc} 
C& E_1 & F_0 & F_{1}\\
 \left( \begin{array}{cccc}
a & 0 & 0&0 \\
0 & b & 0 &0  \\
0 & 0 & c &0\\
0&0&0& \frac{1}{abc}  \end{array} \right) &
 \left( \begin{array}{cccc}
a & 0 & 0&0 \\
0 & b & c &0  \\
0 & 0 & b &0\\
0&0&0& \frac{1}{ab^2}  \end{array} \right)  &
 \left( \begin{array}{cccc}
a & b & 0&0 \\
0 & a & 0 &0  \\
0 & 0 & \frac{1}{a} &c\\
0&0&0& \frac{1}{a}  \end{array} \right) &
\left( \begin{array}{cccc}
a & b & c&0 \\
0 & a & b &0  \\
0 & 0 & a &0\\
0&0&0& \frac{1}{a^3}  \end{array} \right)  
\end{array}}
$$

$$\scriptsize{ \begin{array}{cccc}
F_2 & F_3 & N_{1,  \beta}  &  N_{2} \\
\left( \begin{array}{cccc}
a & b & c&0 \\
0 & a & 0 &0  \\
0 & 0 & a &0\\
0&0&0& \frac{1}{a^3}  \end{array} \right) &
 \left( \begin{array}{cccc}
a & 0 & c&0 \\
0 & a & b &0  \\
0 & 0 & a &0\\
0&0&0& \frac{1}{a^3}  \end{array} \right) &
 \left( \begin{array}{cccc}
1 & a & b&c \\
0 & 1 &   a &  \beta b  \\
0 & 0 & 1 &  \beta a\\
0&0&0& 1  \end{array} \right) &
  \left( \begin{array}{cccc}
1 & a & b&c \\
0 & 1 & a &0  \\
0 & 0 & 1 &0\\
0&0&0& 1  \end{array} \right) 
\end{array}}$$
$$\scriptsize{\begin{array}{cccc}
 N_{3} &N_{4, \alpha}& N_5 & N_6 \\
 \left( \begin{array}{cccc}
1 & 0 & 0&c \\
0 & 1 & a &b  \\
0 & 0 & 1 &a\\
0&0&0& 1  \end{array} \right) &
  \left( \begin{array}{cccc}
1 & a & b&c \\
0 & 1 & 0 &\alpha a  \\
0 & 0 & 1 &b \\
0&0&0& 1  \end{array} \right) &
 \left( \begin{array}{cccc}
1 & 0 & b&c \\
0 & 1 & a &b  \\
0 & 0 & 1 &0\\
0&0&0& 1  \end{array} \right) &
\left( \begin{array}{cccc}
1 & a & 0&c \\
0 & 1 & 0 &0  \\
0 & 0 & 1 &b\\
0&0&0& 1  \end{array} \right) 
\end{array}}$$
$$\scriptsize{\begin{array} {cc} 
N_7 & N_8\\
 \left( \begin{array}{cccc}
1 & 0 & 0&c \\
0 & 1 & 0 &b  \\
0 & 0 & 1 &a\\
0&0&0& 1  \end{array} \right) &
 \left( \begin{array}{cccc}
1 & a & b&c \\
0 & 1 & 0 &0  \\
0 & 0 & 1 &0\\
0&0&0& 1  \end{array} \right) 
\end{array}} $$
where $a,b,c$ run over $\QQ_p$ and  $\alpha, \beta \in \QQ_p$ are fixed.  
 The number of conjugacy classes of limits of $C$ is bounded below by $ 12 +  \square +\mathcal{Q}_4 $  and bounded above by $12 + \square + \mathcal{Q}_8$. 
\end{proposition} 

Iliev and Manivel prove there are 14 conjugacy classes of 3-dimensional abelian subalgebras in $\mathfrak{sl}(4,\CC)$, see \cite{IM}. This list is the same as ours, with only one representative for each family of conjugacy classes, 
 since $\CC$ is algebraically closed.

  


\begin{lemma}\label{N4}  
\begin{enumerate}
\item[a)] The groups $N_{1,\alpha}, N_{1,\beta}$ are conjugate in $\SL(4,\QQ_p)$ if and only if $\alpha,\beta$ are in the same square-class.
\item[b)] If  the groups $N_{4, \alpha}$ and $N_{4, \beta}$ are conjugate in $\SL(4,\QQ_p)$ then $\alpha$ and $\beta$ are in the same coset of $\QQ_p ^* / {\QQ_p ^*}^4$. 
 If $\alpha$ and $\beta$ are in the same coset of $\QQ_p ^* / {\QQ_p ^*}^8$ then the groups $N_{4, \alpha}$ and $N_{4, \beta}$ are conjugate in $\SL(4,\QQ_p)$.  
\end{enumerate}  
\end{lemma}

\begin{proof}   
 \begin{enumerate}
\item[a)] The proof of the necessary condition is very similar to the one of lemma \ref{Nalpha}, we just sketch the idea. Assume that $N_{1,\alpha}, N_{1,\beta}$ are conjugate by some $T\in \SL(4,\QQ_p)$; then the corresponding Lie algebras $\mathfrak{n}_{1,\alpha},\mathfrak{n}_{1,\beta}$ are conjugate by $T$. Writing $\langle e_1\rangle =\cap_{X\in\mathfrak{n}_{1,\alpha}}\ker(X), \langle e_1,e_2\rangle =\cap_{X\in\mathfrak{n}_{1,\alpha}}\ker(X^2), \langle e_1,e_2,e_3\rangle= \cap_{X\in\mathfrak{n}_{1,\alpha}}\ker(X^3)$, we see that the subspaces $\langle e_1\rangle, \langle e_1,e_2\rangle, \langle e_1,e_2,e_3\rangle$ are $T$-invariant, i.e., $T$ is upper triangular. A direct computation expressing that $T$ conjugates $\mathfrak{n}_{1,\alpha}$ into $\mathfrak{n}_{1,\beta}$, then gives $\alpha, \beta$ in the same square-class. Conversely, if $\alpha, \beta$ are in the same square class, $N_{1, \alpha }$  is conjugate to $N_{1, \beta}$   by the diagonal matrix $Diag ((\frac{\beta}{\alpha})^{-\frac{1}{2}}, 1 ,  (\frac{\beta}{\alpha})^{\frac{1}{2}}, 1)$.
\item[b)] If $N_{4,\alpha}$ and $N_{4,\beta}$ are conjugate by $T\in SL(4,\QQ_p)$, then so are the Lie algebras $\mathfrak{n}_{4,\alpha},\mathfrak{n}_{4,\beta}$, where 
$$\mathfrak{n}_{4,\alpha}=\{X_\alpha=\left( \begin{array}{cccc}
0 & a & b&c \\
0 & 0 & 0 &\alpha a  \\
0 & 0 & 0 &b \\
0&0&0& 0  \end{array} \right): a,b,c\in\QQ_p\}.$$
It will be convenient to write this in a more compact form. Set ${\bold a}=\left(\begin{array}{c}a \\b\end{array}\right)\in\QQ_P^2$, and consider the diagonal 2-by-2 matrix $D_\alpha=Diag(\alpha,1)$. Then:
$$\mathfrak{n}_{4,\alpha}=\{X_\alpha=\left( \begin{array}{ccc}
0 & {\bold a}^t &c \\
0 & {\bold 0} &D_\alpha {\bold a}  \\
0&0&0  \end{array} \right): {\bold a}\in\QQ_p^2,c\in\QQ_p\},$$
where ${\bold 0}$ denotes the 2-by-2 zero matrix. Since $\langle e_1\rangle = \cap_{X\in\mathfrak{n}_{4,\alpha}}\ker(X)$ and $\langle e_1,e_2,e_3\rangle = \cap_{X\in \mathfrak{n}_{4,\alpha}} \ker(X^2)$, we see that $\langle e_1\rangle$ and $\langle e_1,e_2,e_3\rangle$ are $T$-invariant subspaces, so that $T$ has the form
$$T=\left( \begin{array}{ccc}
\lambda & {\bold x}^t & \star \\
0 & A &{\bold y}  \\
0&0& \frac{1}{\lambda\det(A)}  \end{array} \right)$$
for some $\lambda\in\QQ_p^*,\bold{x,y}\in\QQ_p^2, A\in GL(2,\QQ_p)$. Then, by direct computation:
$$TX_\alpha T^{-1}=\left( \begin{array}{ccc}
0 & \lambda{\bold a}^t A^{-1} &\star \\
0 & {\bold 0} &\lambda\det(A)AD_\alpha {\bold a}  \\
0&0&0  \end{array} \right)$$
Setting $\bold{a'}:=\lambda(A^t)^{-1}a$, we get
$$TX_\alpha T^{-1}=\left( \begin{array}{ccc}
0 & \bold{a'} &\star \\
0 & {\bold 0} &\det(A)AD_\alpha A^t \bold{a'} \\
0&0&0  \end{array} \right)$$
Expressing that $TX_\alpha T^{-1}$ is in $\mathfrak{n}_{4,\beta}$, we get $\det(A)AD_\alpha A^t\bold{a'}=D_\beta\bold{a'}$ for every $\bold{a'}\in\QQ_p^2$, i.e.
$$\det(A)AD_\alpha A^t=D_\beta.$$
Taking determinants on both sides we get $\det(A)^4\alpha=\beta$, so that $\alpha$ and $\beta$ are in the same class modulo 4-th powers.

 Conversely, if $\alpha$ and $\beta$ are in the same coset of $\QQ_p ^* / {\QQ_p ^*}^8$ then $N_{4, \alpha}$ and $N_{4, \beta}$ are conjugate by the diagonal matrix $\textrm{diag} (1, (\frac{\beta}{\alpha})^{\frac{3}{8}}, (\frac{\beta}{\alpha})^{-\frac{1}{8}},(\frac{\beta}{\alpha})^{-\frac{1}{4}})$.
\end{enumerate}
 \end{proof}

 \begin{proof}[Proof of \ref{sl4}] 
 Just as in the proof of Proposition \ref{sl2} it is sufficient to consider only limits contained in the Borel subgroup product with $\mu_4$.  Using the same arguments as in Proposition \ref{sl3}, it is possible to produce a sequence of conjugating matrices to each of the limits of the Cartan as follows.  
 $$\scriptsize{\begin{array}{cccc} 
E_1 & F_0& F_1 & F_2 \\
\left( \begin{array} {cccc} 
1& 0 & 0& 0 \\
0 & 1 &p^ {-n} & 0 \\
0& 0& 1 &0\\
0 & 0& 0& 1
\end{array} \right) &
\left( \begin{array} {cccc} 
1& p^{-n} & 0& 0 \\
0 & 1 & 0 & 0 \\
0& 0& 1 &p^{-n}\\
0 & 0& 0& 1
\end{array} \right) & 
\left( \begin{array} {cccc} 
1 &p^{- n} & \frac{1}{2} p^{-2n} & 0 \\
0 & 1 & p^{-n} & 0 \\
0& 0&1 &0\\
0 & 0& 0& 1
\end{array} \right) &
\left( \begin{array} {cccc} 
1& p^{-n} & p^{-n}& 0 \\
0 & 1 & 0 & 0 \\
0& 0& 1 &0\\
0 & 0& 0& 1
\end{array} \right) 
  \end{array} }
  $$
    $$\scriptsize{\begin{array}{cccc} 
F_3 & N_{1, \beta} & N_2 & N_3 \\
\left( \begin{array} {cccc} 
1& 0 & p^{-n}& 0 \\
0 & 1 & p^{-n} & 0 \\
0& 0& 1 &0\\
0 & 0& 0& 1
\end{array} \right) &
\left( \begin{array} {cccc} 
\beta & p^{-n} &  \frac{1}{2} p^{-2n}  & \frac{1}{6} p^{-3n} \\
0 & 1 &   1 p^{- n} &  \frac{1}{2} p^{-2n}\\
0& 0& 1 &p^{-n}\\
0 & 0& 0&  \frac{1}{\beta}
\end{array} \right) & 
\left( \begin{array} {cccc} 
1 & p^{-n} & \frac{1}{2} p^{-2n} &p^{- n} \\
0 & 1 & p^{-n} & 0 \\
0& 0& 1 &0\\
0 & 0& 0& 1
\end{array} \right) &
\left( \begin{array} {cccc} 
1& 0 & 0&p^{- n} \\
0 & 1 &p^{ -n} & \frac{1}{2} p^{-2n} \\
0& 0& 1 &p^{-n}\\
0 & 0& 0& 1
\end{array} \right) 
  \end{array} }
  $$
      $$\scriptsize{\begin{array}{ccc} 
N_{4, \alpha} & N_5 & N_6 \\
\left( \begin{array} {cccc} 
1& p^{-n} & p^{-n}& p^{-2n} \\
0 & 1 & 0 &\alpha p^{- n} \\
0& 0& 1 &p^{-n}\\
0 & 0& 0& 1
\end{array} \right) &
\left( \begin{array} {cccc} 
1& 0 & p^{-n}& p^{-n} \\
0 & 1 & p^{-n} &p^{- n} \\
0& 0& 1 &0\\
0 & 0& 0& 1
\end{array} \right) &
\left( \begin{array} {cccc} 
1& p^{-n} & 0&p^{- n} \\
0 & 1 & 0 &0 \\
0& 0& 1 &p^{-n}\\
0 & 0& 0& 1
\end{array} \right) 
  \end{array} }
  $$
        $$\scriptsize{\begin{array}{cc} 
N_7 &N_8 \\
\left( \begin{array} {cccc} 
1& 0 & 0& p^{- n} \\
0 & 1 & 0 & p^{-n} \\
0& 0& 1 &p^{-n}\\
0 & 0& 0& 1
\end{array} \right) &
\left( \begin{array} {cccc} 
1& p^{-n} & p^{-n}& p^{-n} \\
0 & 1 & 0 &0 \\
0& 0& 1 &0\\
0 & 0& 0& 1
\end{array} \right). 
  \end{array}}
  $$
  
  Suprenko and Tyskevtich \cite{ST} classify maximal abelian subgroups over an algebraically closed field.  The work here is necessary to descend from the algebraic closure to $\mathbb{Q}_p$, so we need to understand whether or not it is possible to multiply an entry by a scalar in groups which have repeated entries.  Then we need to determine conjugacy classes of groups.  See also \cite{Htt, Leitnercusp} for classifications over $\RR$. 
 
 If $H$ is a hyperbolic Cartan limit, by Section \ref{subsec::hyper_Cartan} $H$ preserves a flat of dimension 1, 2, or 3.   This is the same codimension
as the largest block which has a common eigenvalue.  The proof of Proposition \ref{sl3} may be applied to classify the hyperbolic Cartan limits as the first 6 groups. 
 
If $H$ is an elliptic Cartan limit, by Theorem~\ref{ellip_unip}, $H$ is contained in the unipotent radical of $B {\green \cdot \mu_4}$.  

 One computes that in order for a group to be abelian, entries on diagonals must be zero or scalar multiples of one another.   Given a group of matrices of the form 
$$ \begin{pmatrix}  1& x & a & t\\
0 & 1 & y & b\\
0 & 0 & 1 & z\\
0 & 0 & 0 &1
\end{pmatrix} 
$$
as $x,y,z,a,b,t$ run over $\QQ_p$, one checks that one of the following conditions is necessary for the group
to be abelian.  
 First notice $x, y$ are either zero or scalar multiples of one another, and the same for $y,z$. 
 \begin{description} 
\item[If $y \neq 0$ and $xz \neq 0$] then from the previous relations we get $x,y,z$ are scalar multiples of one another , and either $a=0$, $b=0$ or $a,b$ are scalar multiples of one another.    If $a$ or $b=0$ then we do not get a group.  So we must have that $x,y,z$ are scalar multiples of each other and $a,b$ are scalar multiples of each other, so that after checking the group laws, we see it has the form: 
$$ \begin{pmatrix} 1 & x & a & t \\
0 & 1 & \alpha x & \beta a \\
0 & 0 & 1 & \beta x \\
 0 & 0 & 0 & 1
\end{pmatrix}
$$
where $\alpha, \beta \in \QQ_p$ are fixed scalars and $x,a,t$ run over $\QQ_p$. But this group is conjugate to a group where $\alpha=1$ by $diag( \alpha, 1,1,\frac{1}{\alpha})$.   So we are left with the group $N_{1,  \beta}$ which provides $\square$ conjugacy classes by lemma \ref{N4}.
\item [If $y=0$ and $xz \neq 0$ and $ab \neq 0$] then $x$ is a scalar multiple of $z \textrm { or } b$ and $z$ is a scalar multiple of  $x \textrm{ or } a$.    So we get the groups $N_{4, \alpha}$, whose conjugacy classes are studied in lemma \ref{N4}. 
\item [If $y=0$ and $xz \neq 0$ and $a, b=0$], we get the group $N_6$.
\item[If $y \neq 0$ and $x,z=0$] then $a,b,t$ may take any value.  This is a maximal abelian subgroup which is 4 dimensional, see \cite{ST}.  Since we are interested in 3 dimensional groups, we require some linear relation in $y,a,b,t$.  One checks that the groups is then either conjugate to $N_5$. 

Thus we are left with the cases where some of $x,z,a,b=0$.  
\item[If $x \textrm { or } z =0$] then $ab=0$ or $a$ is a scalar multiple of $b$.   In the case that $a,b$ are scalar multiples, we do not get a group. 
If one of $a,b=0$, then one checks that the group is conjugate to $N_2$ or $N_3$. 
\item[If $a \textrm { or } b=0$] then either $xz =0$ or $x$ is a scalar multiple of $z$.   If one of $a,b$ is zero, and one of $x,z$ is zero, then the group is conjugate to $N_7$ or $N_8$.  If one of $a,b=0$ and $x$ is a scalar multiple of $z$, then again we do not have a group. 
 \end{description} 
 Thus we have run over all possible cases for $x,y,z,a,b,t$, and we have obtained all of the groups in our list.
 \end{proof} 
 
 This also concludes the proof of Theorem \ref{lowdimthm}.

The  spherical building at infinity for $\SL(4,\QQ_p)$ has three kinds of vertices, whose stabilizers in $\SL(4, \QQ_p)$ are the maximal parabolic  subgroups and they correspond to fixing a point, line, or plane, in $\PP(\QQ_p ^3)$.  Intuitively, these correspond to $N_8, N_6,$ and $N_7$ respectively.  The group $N_5$ also fixes the vertex that corresponds to planes in $\PP(\QQ_p ^3)$. 
  
The group $N_3$ preserves the edge between the vertices fixing a plane and line in $\PP(\QQ_p ^3)$.  The group $N_2$ preserves the edge between vertices  fixing a vertex and line in $\PP(\QQ_p ^3)$.  The groups $N_{4, \alpha}$ preserve the edge  between vertices fixing a plane and a point, since $N_{4, \alpha}$ is contained in the unipotent radical of the parabolic which preserves this edge at infinity.  The groups $N_{1, \beta}$ correspond to preserving the triangle between vertices fixing a vertex, line, plane in $\PP(\QQ_p ^3)$.
  
The groups with eigenvalues preserve opposite faces in the spherical building at infinity for  $\SL(4,\QQ_p)$.  Notice they are made by gluing together unipotent conjugacy limits in lower dimensions. 
  
  
  
  \subsection{$\SL (5,\QQ_p)$ and higher}
  
Using the same sorts of arguments as the above, we conjecture that there are again finitely many limits of the Cartan in $\SL(5,\QQ_p)$. Hyperbolic Cartan limit groups are constructed by gluing together elliptic Cartan limit groups from lower dimensions as each block, with matching diagonal entries on each block, see Theorem \ref{rem::hyp_Cartan}. So there will again be finitely many hyperbolic Cartan limit groups.  Elliptic Cartan limit groups must be contained in the unipotent radical of a parabolic  product with $\mu_5$.  The largest such abelian subgroup of any unipotent radical is the image of an abelian representation from $\QQ_p ^6$.  So we conjecture that it is still only possible to fit finitely many conjugacy classes of limits of the Cartan inside this unipotent radical. 

\medskip  
In $\SL(6,\QQ_p)$ there is an abelian subgroup of a unipotent radical which is the image of a representation from $\QQ_p ^9$.  Computing conjugacy classes of limits of the Cartan in $\SL(6,\QQ_p)$  is open.   For $n \geq 7$, the next section shows there are infinitely many conjugacy limits of the Cartan up to conjugacy. 
  
\begin{remark}  
It is a natural question to ask which limits of the Cartan can limit to others.  The incidence geometry of the spherical building answers this question for us nicely. 

Consider first the case of a hyperbolic Cartan limit $H$. Then $H$ stabilizes a flat in the Bruhat--Tits building. By taking a limit of a sequence of conjugates of $H$ it is then clear that the dimension of the corresponding stabilized flat cannot increase.  So hyperbolic Cartan limits can limit to  either hyperbolic Cartan limits which stabilize lower dimensional flats, or to elliptic Cartan limits.

Recall the unipotent radical of a parabolic subgroups decreases as the dimension of the face in the spherical building at infinity decreases, see Example \ref{ex:levi_decom}.  (We mean there are less blocks in the unipotent radical, and the size of the blocks is smaller.) If $H$ is an elliptic Cartan limit  we expect $H$ to be contained in a unipotent radical $U^{I}$,  which is the unipotent radical stabilizing the face $\mathcal{F}^I$. Then a limit of a sequence of conjugates of $H$ must be contained in a unipotent radical whose corresponding face in the spherical building at infinity is contained in $\mathcal{F}^I$.

Notice that this corrects the digraph of limit groups in \cite{IM}. 
\end{remark}

  
  \section{An Infinite Family of NonConjugate Limits}\label{dim_cl}

In this section we adapt arguments due to Haettel/ Iliev-Manivel and Leitner.  Iliev and Manivel study $Red(n)$, the  Zariski closure of $Cart( \mathfrak{sl}(n,\RR))$ in the Grassmannians $Grass( n-1, \mathfrak{sl}(n,\RR))$ endowed with the Zariski topology.  

 Recall by Proposition \ref{chab_gras} that the Chabauty topology  on  $\mathcal{S}(\mathfrak{sl}(n,\QQ_p), n-1)$ is compatible with the topology of the Grassmannians $Grass( n-1, \mathfrak{sl}(n,\QQ_p))$. Define $Ab(n-1) \subset Grass( n-1, \mathfrak{sl}(n,\QQ_p))$ to be the set of abelian subalgebras of dimension $n-1$.

By results of \cite{Leitner}, we know that over $\RR$ we have $\overline{Cart ( \mathfrak{sl}(n,\RR))}^{Ch} = Ab(n-1)$ for $n \leq 4$ and $\overline{Cart ( \mathfrak{sl}(n,\RR))}^{Ch} \subsetneq Ab(n-1) $ for $n \geq 5$.   The same examples used in \cite{Leitner} to show $\overline{Cart ( \mathfrak{sl}(n,\RR))}^{Ch} \subsetneq Ab(n-1)$ for $n \geq 5$ may be adapted to $\QQ_p$ as the arguments use only linear algebra. 

 \begin{example}[Abelian subgroups which  are not limits of the Cartan]\label{abelnotlim}
Consider the image of the representation $\QQ_p ^4 \to \SL(5, \QQ_p)$ given by 
$$ (a,b,c,d) \mapsto 
\bpmat 1 & a & 0 & \frac{a^2}{2} & b\\
0 & 1 & 0 & a & 0 \\
0 & 0 & 1 & c & d \\
0 & 0 & 0 & 1 & 0 \\
0  & 0 & 0 & 0 & 1
\epmat.$$
This is an abelian subgroup of $\SL (5,\QQ_p)$.  

A \textbf{flat subgroup} 
 is a subgroup of $SL(n,\QQ_p)$ which is the intersection of a vector subspace of $\QQ_p^{n^2}$ with $\SL(n,\QQ_p) \subset \textrm{End}(\QQ_p ^n)$. The diagonal Cartan subgroup  $C$ is flat. 
Conjugacy is a linear map, so it preserves this property, and Chabauty limits of conjugates of $C$ are also flat, by Proposition \ref{chab_gras}.  The group in this example is not flat, and so - although 4-dimensional - it cannot be a limit of $C$. 
Similar examples as above can be constructed for $n \geq 6$. 
 \end{example}  

\begin{theorem}\label{ablim} For $n \geq 5$, we have $\overline{Cart( \mathfrak{sl}(n,\QQ_p)})^{Ch}$ is a proper subset of $Ab(n-1)$. 
\end{theorem} 

\begin{proof}  Extend example \ref{abelnotlim}. 
\end{proof} 

Iliev and Manivel and Haettel give a counting argument which proves Theorem \ref{ablim} for $n \geq 7$ over $\mathbb{C}$ and $\mathbb{R}$ respectively.  They count the dimension of the space of subalgebras isomorphic to $(\mathbb{R}^{n-1}, +)$ and show is  cubic in $n$, compared with the dimension of the space of limits of the Cartan subalgebra, which is quadratic in $n$.  Leitner's examples gave the first explicit examples of subalgebras isomorphic to $(\mathbb{R}^{n-1},+)$ which are not limits of the diagonal Cartan  subalgebra. 


Leitner \cite{Leitner} also gives a lower bound on the covering dimension of $\overline{Cart ( \SL(n, \RR)})^{Ch} /conjugacy$ over $\RR$. 
To extend this to $\mathbb{Q}_p$ a good notion of dimension is needed, and we intend to explore this in future work.

 Let $G$ be a group which acts on the projective space $\mathbb{P}(K^n)$ over a field $K$ and let $A,B \leq G$ be subgroups.  Every orbit of $A$ has a closure which spans a projective subspace.  The \textbf{dimension of an orbit closure} is the dimension of this projective subspace. The set of \textbf{orbit closures} of $A$ is the set of all closures of all orbits of $A$. 
 If $A$ and $B$ are conjugate, then the orbit closures of $A$ and $B$ are projectively equivalent, i.e., there is a projective transformation taking the orbit closures of $A \curvearrowright \mathbb{P}(K^n)$ to the orbit closures of $B$. 
  Leitner produces a continuum of groups with non-conjugate orbit closures.  We adapt the same argument from \cite{Leitner}. 

\begin{proposition}  For $n \geq 7$, there are infinitely many nonconjugate limits of $C$ in $\SL(n, \mathbb{Q}_p)$. 
\end{proposition} 
\begin{proof}  

We give a sketch of proof for $n=7$, which contains all the main ideas and significantly reduces notation.  The details for $n \geq 7$ may be read in \cite{Leitner}, and nothing about the topology of $\RR$ is used, so the same proof will work in the projective space over $\QQ_p$.

Let $\alpha \in \QQ_p -\{ 0,1,2 \}$ be fixed.  Consider the homomorphism $\rho_\alpha: \QQ_p ^6 \to  \SL(7, \mathbb{Q}_p)$ and the sequence $
\{p_m\}_{m \in \NN}$
$$\rho_\alpha (a,b,c,d,e,f):=  
\bpmat 1 & 0 & 0 & 0& 0&  a & 0\\
0 & 1 & 0 & 0 & 0 & b & b \\
0 & 0 & 1 & 0 & 0 & c & 2c \\
0 & 0 & 0 & 1 & 0 & d & \alpha d \\
 0 & 0 & 0 & 0 & 1 & e & f \\
 0 & 0 & 0 & 0 & 0 & 1 & 0 \\
 0 & 0 & 0 & 0 & 0 & 0 & 1\\
\epmat 
\qquad 
 p_m:= \bpmat 1& 0 & 0 & 0 &0 & m & 0\\
0 & 1 & 0 & 0 & 0 & m& m\\
0 & 0 & 1 & 0 & 0 & m & 2m \\
0 & 0 & 0 & 1 & 0 & m & \alpha m \\
0 & 0 & 0 & 0 & 1 & m ^2 & m^2 \\
0 & 0 & 0 & 0 & 0 & 1 & 0\\
0 & 0 & 0 & 0 &  0 & 0 & 1
\epmat. 
$$
Then the image of $\rho_\alpha$ is a group $L_\alpha$.  We see that $L_\alpha$ is a limit of the Cartan under conjugacy by conjugating $C$ by the sequence $\{p_m\}_{m \in \NN}$ of matrices and taking a limit as $m \to \infty$ in the Chabauty topology.  

We want to show $L_\alpha$ is not conjugate to $L_\beta$, for  
$$\beta \not \in \mathcal{UC} \{ 0,1,2, \alpha \} = \left\{\frac{2(\alpha-1)}{\alpha} , \frac{\alpha}{2(\alpha-1)}, \frac{\alpha}{2-\alpha} , \frac{2-\alpha}{\alpha}, \frac{2(\alpha-1)}{\alpha-2}, \frac{\alpha-2}{2(\alpha-1)} \right\}. $$ 
Here $\mathcal{UC} \{0,1,2, \alpha\}$ is the set of all possible cross ratios of the four points $\{[1:0], [1:1], [1:2], [1:\alpha] \}$ in any order.

We follow Leitner's argument and show the orbits of $L_\alpha$ acting on the projective space $\mathbb{P}(\QQ_p ^7)$ are not projectively equivalent to the orbit closures of $L_\beta$. The orbit closures of $L_\alpha$ are a fixed four dimensional space, surrounded by a sheaf of 5 dimensional spaces, some of which break down further.
We count the dimension 
of an orbit of a point $x \in \mathbb{P}(\QQ_p ^7)$ under $L_\alpha$.  There are three cases.    


Case 1:  $x \in \langle e_1, ..., e_5 \rangle $.  Then $L_\alpha$ acts as the identity, and the dimension of an orbit is 0. 

Case 2: $x \in \langle e_1, ..., e_5 , t e_6 + e_7  \rangle $, and $t \not \in \{ 0,1,2,\alpha\}$ then the orbit is 5 dimensional. 

Case 3: $x \in \langle e_1, ..., e_5 , t e_6 + e_7  \rangle $, and $t \in \{ 0,1,2,\alpha\}$, then the orbit is 4 dimensional. 


We can project out the space $\langle e_1, ..., e_5 \rangle $ and onto the projective line $\langle e_6, e_7 \rangle$.  This gives us four special points on the line, $\{[1:0], [1:1], [1:2], [1:\alpha] \}$.
The unordered cross ratio is a projective invariant of these four points.  
So the sets of points $\{[1:0], [1:1], [1:2], [1:\alpha] \}$ and $\{[1:0], [1:1], [1:2], [1:\beta] \}$ are projectively equivalent if and only if they have the same unordered cross ratio, and this proves the claim. 

Leitner generalizes this proof to groups with a matrix of coefficients in the upper right.   The dimension count comes from counting the degrees of freedom in the matrix after applying the unordered generalized cross ratio.  All of the details can be found in \cite{Leitner}. 
\end{proof}

\end{document}